\documentclass[11pt, notitlepage]{article}
\usepackage{amssymb,amsmath,comment}
\catcode`\@=11 \@addtoreset{equation}{section}
\def\thesection{\arabic{section}}

\def\theequation{\thesection.\arabic{equation}}
\catcode`\@=12
\usepackage{colortbl}
\usepackage[mathscr]{eucal}
\usepackage{enumitem}
\usepackage{epsf}
\usepackage{a4wide}
\def\R{\mathbb{R}}

\newcommand{\e}{\epsilon}

\newcommand{\Om} {\Omega}

\newcommand{\De} {\Delta}
\newcommand{\la} {\lambda}

\newcommand{\noi} {\noindent}

\newcommand{\mb} {\mathbb}
\newcommand{\mc} {\mathcal}

\usepackage[all]{xy}
\catcode`\@=11
\def\theequation{\@arabic{\c@section}.\@arabic{\c@equation}}
\catcode`\@=12

\newcommand{\QED}{\rule{2mm}{2mm}}

\newtheorem{Theorem}{Theorem}[section]
\newtheorem{Lemma}[Theorem]{Lemma}
\newtheorem{Proposition}[Theorem]{Proposition}

\newtheorem{Definition}[Theorem]{Definition}
\newtheorem{Example}{Example}

\newcommand{\eqdef}{\stackrel{{\mathrm {def}}}{=}}

\begin{document}

{\vspace{0.01in}}

\title
{ \sc Polyharmonic Kirchhoff problems involving exponential non-linearity of Choquard type with singular weights.}

\author{~~R. Arora,~~J. Giacomoni\footnote{LMAP (UMR E2S-UPPA CNRS 5142) Bat. IPRA, Avenue de l'Universit\'e F-64013 Pau, France. email: rakesh.arora@univ-pau.fr, jacques.giacomoni@univ-pau.fr},~~ T. Mukherjee\footnote{Tata Institute of Fundamental Research(TIFR) Centre for Applicable Mathematics, Bangalore, India. e-mail: tulimukh@gmail.com}~ and ~K. Sreenadh\footnote{Department of Mathematics, Indian Institute of Technology Delhi, Hauz Khaz, New Delhi-110016, India. e-mail: sreenadh@gmail.com} }

\date{}

\maketitle

\begin{abstract}
\noi In this work, we study the higher order Kirchhoff type Choquard equation $(KC)$ involving a critical exponential non-linearity and singular weights. We prove the existence of solution to $(KC)$ using Mountain pass Lemma in light of Moser-Trudinger and singular Adams-Moser inequalities. In the second part of the paper, using the Nehari manifold technique and minimization over its suitable subsets, we  prove the existence of at least two solutions to the  Kirchhoff type Choquard equation $(\mathcal{P_{\la,M}})$ involving convex-concave type non-linearity. \\ \\
\noi \textbf{Key words:} Doubly non local equation, Kirchhoff term, Choquard non-linearity with singular weights, Singular Adams-Moser inequality, Nehari Manifold, Polyharmonic operator.
\vspace{.2cm}\\
\noi \textit{2010 Mathematics Subject Classification:} 35R10, 35R09, 35J60, 35A15, 35J35.

\end{abstract}

\section{Introduction}
The main objective of this paper is to prove the existence of non-trivial weak solution of the following Kirchhoff type Choquard equation with exponential non-linearity and singular weights:
\begin{equation*}
     (KC)
     \left\{
         \begin{alignedat}{2}
             {} -M\left(\int_\Om|\nabla^m u|^{2}~dx\right)\De^m u
             & {}=  \left(\int_{\Om} \frac{F(y,u)}{{ |y|^\alpha} |x-y|^\mu}dy\right)\frac{f(x,u)}{|x|^{\alpha}}~dx,\;
             && \quad\mbox{ in }\, \Omega ,
             \\
             u=\nabla u=\dots =\nabla^{m-1}u & {}= 0
             && \quad\mbox{ on }\, \partial\Omega,
          \end{alignedat}
     \right.
\end{equation*}
where $m\in \mathbb N$, $n = 2m$, $\mu \in (0,n)$, $ 0<\alpha <\min\{\frac{n}{2}, n-\mu\} $, $\Om$ is a bounded domain in $\mb R^n$ with smooth boundary, $ M:\mb R^+ \to \mb R^+$ and $f: \Omega \times \mb R \to \mb R$ are continuous functions satisfying suitable assumptions specified in details later. The function $F$ denotes the primitive of $f$ with respect to the second variable.\\
We also study the existence of weak solutions of a Kirchhoff type Choquard equation with convex-concave sign changing non-linearity:
 \begin{equation*}
     (\mathcal{P}_{\la, \mathcal{M}})
     \left\{
         \begin{alignedat}{2}
             {} -M\left(\int_\Om|\nabla^m u|^2~dx\right)\De^m u
             & {}=  \lambda h(x)|u|^{q-1} u + \left( \int_{\Omega}\frac{F(u)}{|x-y|^\mu |y|^\alpha}~dy\right)\frac{f(u)}{|x|^\alpha}
             && \ \ \mbox{in }\, \Omega,
             \\
             u= \nabla u= \dots = \nabla^{m-1} u & {}= 0
             && \  \mbox{on }\, \partial\Omega,
             %\\
            % \quad u & > 0 &&\ \ \mbox{in }\, \Omega
          \end{alignedat}
     \right.
\end{equation*}
where $f(u)=u|u|^p \exp(|u|^{\gamma})$, $0 < q< 1,\ 2< p$, $\gamma \in (1, 2)$ and $F(t)=\int_{0}^t f(s)~ds$. In this case, we assume $M(t)=at+b$ where $a, b>0$ and $h \in L^{r}(\Omega)$ where {$r =\frac{p+2}{q+1}$} is such that $h^+ \not\equiv 0$.\\
The boundary value problems involving Kirchhoff term appear in various physical and biological
systems. In $1883$, Kirchhoff observed these kinds of non-local phenomena  in the study of string or membrane vibrations to describe the transversal oscillations, by considering the effect of changes in the length of the string. In the case of Laplacian operator, problems having Kirchhoff term arise from the theory of thin plates and describe the deflection of the middle surface of a p-power like elastic isotropic flat plate of uniform thickness. Precisely, $M(\|u\|^p)$ represents the non-local flexural rigidity of the plate depending continuously on $\|u\|^p$ of the defection $u$ in the presence of non-linear source forces. \vspace{.2cm}\\
Initially in \cite{ACF}, Alves et al. considered the following non-local Kirchhoff problem with Sobolev type critical non-linearities
$$-M\left(|\nabla u|^2 dx \right) \Delta u= \lambda f(x,u) +u^5 \ \text{in}\ \Omega,\;\; u=0 \ \text{on}\ \partial\Omega,$$
where
$\Omega \subset \mathbb{R}^3$ is a bounded domain with smooth boundary and $f$ has sub-critical growth
at $\infty$. Using the Mountain-pass Lemma and the compactness analysis of local Palais-Smale sequences,
authors showed the existence of solutions for large $\lambda$. Corr\"ea and Figueiredo \cite{CF} studied the existence of positive solutions
for Kirchhoff equations involving p-laplacian operator with critical or super critical Sobolev type non-linearity. Later on, Figueiredo \cite{Fi} and Goyal et al. \cite{GMS} studied the  Kirchhoff problem with critical exponential growth non-linearity. Recently in \cite{AGMS}, authors have studied the Kirchhoff equation with exponential non-linearity of Choquard type and established the existence and multiplicity of solutions. \vspace{.2cm}\\
Problems involving polyharmonic operators and polynomial type critical growth non-linearities have been broadly studied by many authors till now, see \cite{Ga,GWZ,Gr,PS} for instance. In \cite{PS}, Pucci and Serrin have considered the following critical growth problem in an open ball of $\mathbb{R}^n$:
\begin{equation*}
(-\Delta)^K u= \lambda u + |u|^{s-1} u \ \text{in}\  B,\;\; u=Du= \dots = D^{K-1} u=0 \ \text{on}\ \partial B,
\end{equation*}
where $K-1 \in \mathbb{N},\ s = \frac{n+2K}{n-2K}$, $n > 2K$.  They showed the existence of nontrivial radial solutions to the above problem in a suitable range of $\lambda$. In \cite{Ga}, Gazzola studied the existence result for the same critical exponent polyharmonic problem as above while considering a lower order perturbation term having a subcritical growth at infinity instead of '$\la$u'. We cite   \cite{GS,LO,LL,ZD} and references therein for existence results on polyharmonic equations with exponential type non-linearity.\\ The multiplicity of solutions for elliptic partial differential equations involving polynomial type non-linearity and sign-changing weight functions has been extensively studied in
\cite{AH,BZ,DP,Wu,Wu1}) using the Nehari manifold technique. In \cite{CKW}, authors studied the existence of multiple positive solutions of Kirchhoff type problems with convex-concave type polynomial non-linearities having sub-critical growth using fibering map analysis and the Nehari manifold method. \\

\noi At this point, we remark that the study of  polyharmonic Kirchhoff problems involving the exponential type Choquard non-linearity with singular weights was completely open until now. So our article establishes new results for such problems. We point out that the analysis we use here is completely new for this class of problems with critical growth. The salient feature of our problem $(KC)$ is its doubly-non-local structure due to the presence of non-local Kirchhoff as well as Choquard term which makes the equation $(KC)$ no longer a pointwise identity. The doubly non-local nature induces additional mathematical difficulties in the use of classical methods of non-linear analysis. Additionally, we explore the existence and multiplicity results for these kind of problems in the presence of singular weights.\\
The vectorial polyharmonic operator $\Delta^{m}_{\frac{n}{m}}$ is defined by induction as
\begin{equation*}
\Delta^m_{\frac{n}{m}} u=\left\{
\begin{split}
&\nabla.\{\Delta^{j-1}(|\nabla \Delta^{j-1} u|^{\frac{n}{m}-2} \nabla \Delta^{j-1}u)\} \; \ \  \text{if}\ m=2j-1, \\
&\Delta^{j}(|\Delta^j u|^{\frac{n}{m}-2} \Delta^j u) \;\ \ \ \ \ \ \ \ \ \ \ \ \ \ \ \ \ \ \ \ \ \ \text{if}\ m=2j.
\end{split}
\right.
\end{equation*}
In our case, $\frac{n}{m}=2$, the symbol $\nabla^m u $ denotes the $m^{\text{th}}$-order gradient of $u$ and is defined as
\begin{equation*}
\nabla^m u=\left\{
\begin{split}
&\nabla \Delta^{(m-1)/2}u \; \  \text{if}\ m \ \text{is odd},  \\
&\Delta^{m/2} u \;\ \ \ \ \ \ \ \ \text{if}\ m \ \text{is even}
\end{split}
\right.
\end{equation*}
where $\Delta$ and $\nabla$ denotes the usual Laplacian and gradient operator respectively and also $\nabla^m u. \nabla^m v$ denotes the product of two vectors when $m$ is odd and product of two scalars when $m$ is even.

\noi The study of elliptic equations with critical exponential type non-linearity in higher dimensions is motivated by the following Adams-Moser inequality \cite{Ad} and singular Adams-Moser inequality \cite{LL1}. We denote
\begin{equation*}
\zeta_{m,n}=\left\{
\begin{split}
&\frac{n}{\omega_{n-1}}\left(\frac{\pi^{n/2}2^m \Gamma\left(\frac{m+1}{2}\right)}{\Gamma\left(\frac{n-m+1}{2}\right)}\right)^{\frac{n}{n-m}} \; \ \ \text{when}\ m \ \text{is odd},\\
&\frac{n}{\omega_{n-1}}\left(\frac{\pi^{n/2}2^m \Gamma\left(\frac{m}{2}\right)}{\Gamma\left(\frac{n-m}{2}\right)}\right)^{\frac{n}{n-m}} \; \ \ \ \ \  \text{when}\ m \ \text{is even},
\end{split}
\right.
\end{equation*}
where $\omega_{n-1}=$ $(n-1)$-dimensional surface area of $\mb S^{n-1}$. We use the framework of Hilbert space  $W_0^{m, 2}(\Omega)$ (or $H^m_0(\Om)$) equipped with the natural Banach norm $\|u\|\eqdef\left(\int_\Om|\nabla^m u|^2~dx\right)^{\frac{1}{2}}$ associated to the inner product
\[\langle u,v\rangle = \int_\Om \nabla^m u. \nabla^m v~dx.\] Then, we have the following important theorems.
%{\color{yellow} Reference of both inequalities and combine them is possible.}
\begin{Theorem}{\textbf{(D. Adams, 1988)}}\label{TM-ineq}
Let $\Omega$ be a smooth bounded domain in $\mathbb{R}^n$ and m is a  positive integer satisfying $m<n$. Then for all $0 \leq \zeta \leq \zeta_{m,n}$ we have
\[\sup_{u \in W_0^{m, \frac{n}{m}}(\Omega), \|u\|\leq 1}\int_\Om \exp(\zeta |u|^{\frac{n}{n-m}})~dx < \infty\]
where $\zeta_{m,n}$ is sharp.
\end{Theorem}
\begin{Theorem}{\textbf{(Adams-Moser)}}\label{TM-ineq1}
For $0< \alpha <n$, $\Omega$ be a smooth bounded domain in $\mathbb{R}^n$ and m is a positive integer satisfying $m<n$ we have
\begin{equation}\label{AM}
\sup_{u \in W_0^{m, \frac{n}{m}}(\Omega), \|u\|\leq 1}\int_\Om \frac{\exp(\kappa |u|^{\frac{n}{n-m}})}{|x|^\alpha}~dx < \infty
\end{equation}
for all $0 \leq \kappa \leq \kappa_{m,n} = \left(1-\frac{\alpha}{n}\right)\zeta_{m,n}$, where $\kappa_{m,n}$ is sharp. 
\end{Theorem}
We recall that the embedding 
$$W_0^{m, \frac{n}{m}}(\Omega) \ni u \mapsto \frac{e^{|u|^\beta}}{|x|^\alpha} \in L^1(\Omega)$$ is compact for all $\beta \in[1, \frac{n}{n-m})$ and continuous for $\beta= \frac{n}{n-m}.$ In our case $\frac{n}{n-m}=2$.\\
Now we state the doubly-weighted Hardy-Littlewood-Sobolev inequality proved in \cite{SW}.
\begin{Proposition}\label{HLS}
(\textbf {Doubly weighted Hardy-Littlewood-Sobolev inequality}) Let $t,r>1$ and $0<\mu<n $ with $ \alpha+ \beta \geq 0$, $\frac{1}{t}+\frac{\mu+\alpha+\beta}{n}+\frac{1}{r}=2$, $\alpha < \frac{n}{t^\prime}$, $\beta < \frac{n}{r^{\prime}}$ $f \in L^t(\mathbb R^n)$ and $h \in L^r(\mathbb R^n)$, where $t^\prime$ and $r^\prime$ denotes the H\"{o}lder conjugate of $t$ and $r$ respectively. Then there exists a constant
%{\color{yellow} What is value of that constant}
$C(\alpha, \beta, t,n,\mu,r)>0$ which is independent of $f,h$ such that
 \begin{equation}\label{HLSineq}
 \int_{\mb R^n}\int_{\mb R^n} \frac{f(x)h(y)}{|x-y|^{\mu} |y|^\alpha |x|^{\beta}}~dxdy \leq { C(\alpha, \beta, t, n, \mu, r)}\|f\|_{L^t(\mb R^n)}\|h\|_{L^r(\mb R^n)}.
 \end{equation}
 \end{Proposition}
Throughout the next sections, we assume the following conditions on $M$ and $f$. The function $ M : \mb R^+ \to \mb R^+$ is a continuous function satisfying the following conditions:
\begin{enumerate}
\item[(m1)] There exists $M_0>0$ such that $M(t)\geq M_0$ and
$\mathcal{M}(t+s)\geq \mathcal{M}(t)+\mathcal{M}(s), \; \text{for all}\; t,s\geq 0$
where $\mathcal{M}(t)= \int_0^t M(s) ~ds$ is the primitive of the function $M$ vanishing at $0$.
\item[(m2)] There exist constants $b_1,b_2>0$ and $\hat t>0$ such that for some $k\in \mb R$
\[M(t)\leq b_1+b_2t^k,\;\text{for all}\; t \geq \hat t.\]
\item[(m3)] The function $\frac{M(t)}{t}$ is non-increasing for $t>0$.
\end{enumerate}
Using (m3), one can easily deduce that the function
 \[(m3)^\prime \quad \quad \quad \frac{1}{2}\mathcal{M}(t)-\frac{1}{\theta}M(t)t \;\text{is non-negative and non-decreasing for}\; t\geq 0 \;\text{and}\; \theta \geq 4.\]
\begin{Example}
An example of a function satisfying (m1), (m2) and (m3) is $M(t)= M_0+ bt^\beta$ where $M_0,>0$, $\beta<1$ and $b\geq0$. Also $M(t)= M_0+\log(1+t)$ with $M_0\geq 1$ verifies (m1)-(m3).
\end{Example}
The function $f:\Om \times \mb R\to \mb R$ which governs the Choquard term is given by $f(x,t)=h(x,t)\exp(t^{2}),$ where $h \in C( \Om \times \mb R)$ satisfies the following growth conditions:
\begin{itemize}
\item[(h1)] $h(x,t)=0$ for all $t \leq 0$ and $h(x,t)>0$ for $t>0$.
%{\color{red} Don't we need $h(x,0) >0?$}.
 \item[(h2)]  For any $\e>0$, $\lim\limits_{t \to \infty}\sup_{x \in \bar \Om}h(x,t)\exp(-\e t^{2})=0$ and $\lim\limits_{t \to \infty}\inf_{x \in \bar \Om}h(x,t)\exp(\e t^{2})=\infty$.
 %Really required? and If required can't we replace $\epsilon$ with $\delta$ (we need symmtricity)? }.
    \item[(h3)] There exists $\ell > \max\{1, k+1\}$
    %{ We need $\ell > \frac{n}{m}-1$ and $\ell > \frac{n(k+1)}{2m}$}
     such that $\frac{f(x,t)}{t^{\ell}}$ is increasing for each $t>0$ uniformly in $x \in \Om$, where $k$ is specified in (m2).
        \item[(h4)] There exist $T, T_0>0$ and $\gamma_0 >0 $ such that $0<t^{\gamma_0}F(x,t)\leq T_0 f(x,t)$ for all $|t|\geq T$ and uniformly in $x \in \Omega$.
\end{itemize}
%Is is easy to see that (h3) implies that for each $K_1>0$ it holds
%\[(h3)^\prime \quad \quad \quad K_1 F(x,t)\leq tf(x,t), \; \text{for all}\; (x,t)\in \Om \times \mb R. \]
The condition (h3) implies that { $\frac{f(x,t)}{t}$} is increasing in $t>0$ and $\displaystyle\lim_{t\to 0^+}\frac{f(x,t)}{t}=0$ uniformly in $x \in \Om.$
\begin{Example}
A typical example of $f$ satisfying $(h1)-(h4)$ is $f(x,t)= {t^{\beta+1}\exp(t^p)\exp(t^2)}$ for $t \geq 0$ and $f(x,t)=0$ for $t<0$  where $0\leq p< 2$ and $\beta>l-1$.
\end{Example}
\noindent Furthermore, using ${(h1)-(h3)}$ we obtain that for any $\e>0$, $ r > \beta_0+1$ where { $0 \leq \beta_0 <\ell$ }, there exist constants {$C_1, C_2 >0$ (depending upon $\epsilon, n, m $)} such that for each $x \in \Om$
\begin{equation}\label{kc-1}
0 \leq F(x,t) \leq C_1 |t|^{\beta_0+1}+ C_2 |t|^r\exp((1+\e)t^{2}), \; \text{for all}\; t \in\mb R.
\end{equation}
For any $u \in W_0^{m,2}(\Om)$, by virtue of Sobolev embedding we get that $u \in L^q(\Om)$ for all $q \in [1,\infty)$. This also implies that $F(x,u) \in L^{q}(\Om)\mbox{ for any }q\geq 1$.
The problem $(KC)$ has a variational structure and  the energy functional $\mathcal{J} : W_0^{m,2}(\Om) \to \mb R$ associated to $(KC)$ is given by
\begin{equation}\label{functional}
\mathcal{J}(u)= \frac{1}{2}\mathcal{M}(\|u\|^{2}) - \frac12 \int_\Om \left(\int_{\Omega} \frac{F(y,u)}{ |y|^\alpha|x-y|^\mu}dy\right)\frac{F(x,u)}{|x|^\alpha}~dx.
\end{equation}
%{\color{yellow} Lines about Gateaux Differentiablity of energy functional}\\
 \noindent The notion of weak solution for $(KC)$ is given as follows.
\begin{Definition}\label{def}
A weak solution of $(KC)$ is a function $u \in W^{m,2}_0(\Om)$ such that for all $\varphi \in W^{m,2}_0(\Om)$, it satisfies
\begin{equation}\label{weakfunctional}
M(\|u\|^{2}) \int_\Om \nabla^m u. \nabla^m \varphi ~dx = \int_\Om \left(\int_{\Om}\frac{F(y,u)}{{ |y|^\alpha} |x-y|^{\mu}}dy\right)\frac{f(x,u)}{|x|^\alpha}\varphi ~dx.
\end{equation}
\end{Definition}

\noi In section $2$, we establish the following main result concerning the problem $(KC)$.

\begin{Theorem}\label{kc-mt-1}
Let (m1)-(m3) and (h1)-(h4) holds. Assume in addition
{\begin{equation}\label{h-growth}
\displaystyle \lim_{s\to +\infty} \frac{sf(x,s)F(x,s)}{\exp\left(2 s^{2}\right)} = \infty,\mbox{ uniformly in }x \in \overline{\Om}.
\end{equation}}
 Then the problem $(KC)$ admits a non-trivial weak solution.
\end{Theorem}
In section $3$, we consider the problem $(\mathcal{P}_{\la, \mathcal{M}})$. The energy functional $\mathcal{J}_{\la, M} : W_0^{m,2}(\Omega) \to \mathbb{R}$ associated to the problem $(\mathcal{P}_{\la, \mathcal{M}})$ is defined as
\begin{align*}
\mathcal{J}_{\la, M}(u)= \frac{1}{2} \mathcal{M}(\|u\|^2) - \dfrac{\la}{q+1} \int_{\Omega} h(x) |u|^{q+1}~dx - \dfrac{1}{2} \int_{\Omega} \left( \int_{\Omega} \frac{F(u)}{|x-y|^{\mu} |y|^\alpha}~dy\right) \frac{F(u)}{|x|^\alpha}~dx
\end{align*}
where $F$ and $\mathcal{M}$ are primitive of $f$ and $M$ respectively vanishing at $0$ and $f(s)= s|s|^p \exp(|s|^{\gamma}).$
\begin{Definition}
A function $u \in W_0^{m,2}(\Omega)$ is said to be a weak solution of \ $(\mathcal{P}_{\la, \mathcal{M}})$ if for all $\phi \in W_0^{m,2}(\Omega)$, it satisfies
\begin{align*}
M(\|u\|^2)\int_{\Omega} \nabla^m u. \nabla^m \phi ~dx= \la \int_{\Omega} h(x) |u|^{q-1} u \phi ~dx + \int_{\Omega} \left( \int_{\Omega} \frac{F(u)}{|x-y|^\mu |y|^\alpha}~dy \right) \frac{f(u)}{|x|^\alpha} \phi ~dx.
\end{align*}
\end{Definition}
We prove the following theorem concerning $(P_{\la,\mc M})$.
\begin{Theorem}\label{first}
There exists a $\la_0>0$ such that for $\gamma \in \left(1, 2\right)$ and $\la \in (0, \la_0)$, $(\mathcal{P}_{\la, \mathcal{M}})$ admits atleast two solutions.
\end{Theorem}
Turning to the layout of the paper: In section 2, we prove the existence result (Theorem \ref{kc-mt-1}) for the problem $(KC)$ and in section 3, we prove the existence and multiplicity result (Theorem \ref{first}) for the problem $(\mathcal{P}_{\la,\mc  M}).$

\section{Existence result for $(KC)$}
In this section, we establish the existence of a nontrivial weak solution for the problem $(KC)$. To prove this we study the mountain pass geometry of the energy functional $\mathcal{J}$ and using the properties of the non-local term $M$ and the exponential growth of $f$, we prove that every Palais Smale sequence is bounded. To study the compactness of Palais Smale sequences for $\mathcal{J}$, we obtain a bound for the mountain pass critical level with the help of Adams functions and establish the convergence of weighted Choquard term for Palais-Smale sequences.
\subsection{Mountain pass geometry}
In the following result, we show that the energy functional $\mathcal{J}$ possesses the mountain pass geometry around 0 in the light of Adams-Moser and doubly weighted Hardy-Littlewood-Sobolev inequality.
\begin{Lemma}\label{lemma3.1}
Under the assumptions (m1), (m2) and (h1)-(h3) the following assertions hold:
\begin{enumerate}[label=(\roman*)]
\item there exists $R_0>0, \eta >0$ such that $\mathcal{J}(u) \geq \eta $ for all $u \in W^{m,2}_0(\Om)$ such that $\|u\|= R_0.$
\item there exists a $v \in W_0^{m ,2}(\Omega)$ with $\|v\|>R_0$ such that $\mathcal{J}(v) <0.$
\end{enumerate}

\end{Lemma}
\begin{proof}
Using Proposition \ref{HLS} with $t=r$ and $\beta=\alpha$ and \eqref{kc-1}, we obtain that for any $\e>0$ and $u\in W_0^{m ,2}(\Omega)$, there exist constants $C_i>0$ depending upon $\epsilon$  such that
\begin{equation*}\label{kc-MP1}
\begin{split}
&\int_{\Om}\left(\int_\Om \frac{F(y,u)}{|y|^\alpha|x-y|^{\mu}}dy\right)\frac{F(x,u)}{|x|^{\alpha}}~dx  \leq C(m,\mu,\alpha)\|F(x,u)\|_{L^\frac{2n}{2n-(2 \alpha +\mu)}}^2\\
& \leq \left( C_1 \int_\Om |u|^{\frac{2n(\beta_0+1)}{2n-(2 \alpha +\mu)}} +\ C_2 \int_\Om |u|^{\frac{2rn}{2n-(2 \alpha +\mu)}}\exp\left(\frac{2n(1+\e)}{2n-(2 \alpha +\mu)}|u|^{2} \right) \right)^{\frac{2n-(2 \alpha +\mu)}{n}}\\
 &\leq \left( {C_1} \int_\Om |u|^{\frac{2n(\beta_0+1)}{2n-(2 \alpha +\mu)}} +\ {C_2}  \|u\|^{\frac{2rn}{2n-(2\alpha+ \mu)}} \left( \int_\Om\exp\left(\frac{4n(1+\e)\|u\|^2}{2n-(2 \alpha +\mu)}\left(\frac{|u|}{\|u\|}\right)^{2}\right)\right)^{\frac12} \right)^{\frac{2n-(2 \alpha +\mu)}{n}}
 \end{split}
\end{equation*}
For small $\e>0$ and $u$ such that $\displaystyle\frac{4n(1+\e)\|u\|^{2}}{2n-(2 \alpha +\mu)} \leq \zeta_{m,2m}$, using Theorem \ref{TM-ineq}, we obtain
\begin{equation}\label{x1}
\begin{split}
\int_{\Om}\left(\int_\Om \frac{F(y,u)}{|y|^\alpha|x-y|^{\mu}}dy\right)\frac{F(x,u)}{|x|^{\alpha}}~dx & \leq { C_3} \left( \|u\|^{\frac{2n(\beta_0+1)}{2n-(2 \alpha +\mu)}}  + \|u\|^{\frac{2rn}{2n-(2 \alpha +\mu)}} \right)^{\frac{2n-(2 \alpha +\mu)}{n}}\\
 & \leq {C_4} (\|u\|^{2(\beta_0+1)}  +  \|u\|^{2r}).
 \end{split}
 \end{equation}
Then for $\|u\|<\rho= \left(\frac{\zeta_{m,2m}(2n-(2 \alpha +\mu))}{4n(1+\e)}\right)^{\frac{1}{2}}$,  $(m1)$ and \eqref{x1} gives
%and by above estimates, we deduce that ,
\begin{align*}
\mathcal{J}(u) &\geq M_0\frac{\|u\|^{2}}{2}-   C_4  \|u\|^{2(\beta_0+1)} - C_4 \|u\|^{2r}.
\end{align*}
So we choose $\|u\|= R_0$ small enough so that $\mathcal{J}(u) \geq \eta$ for some $\eta>0$ (depending on $R_0$)  and hence (i) follows. Furthermore $(m2)$ implies that
\begin{equation*}
\mathcal{M}(t)\leq \left\{
\begin{split}
&b_0+b_1t + \frac{b_2t^{k+1}}{k+1},\;k\neq -1\\
&b_0+b_1t + b_2\ln t,\;k= -1
\end{split}
\right.
\end{equation*}
 for $t\geq \hat t$ where
\begin{equation*}
b_0=\left\{
\begin{split}
&\mathcal{M}(\hat t)-b_1\hat t-b_2\frac{\hat t^{k+1}}{k+1},\;k\neq -1,\\
&\mathcal{M}(\hat t) - b_1\hat t - b_2\ln \hat t,\;k= -1.
\end{split}
\right.
\end{equation*}
Under the assumption (h3), there exists $ K_1 \geq \max \{1,k+1 \}$, $c_1, c_2>0$ such that $F(x,s) \geq c_1s^{K_1}-c_2$ for all $(x,s) \in \Om \times [0,\infty)$. Therefore for $v \in W_0^{m ,2}(\Omega)$ such that $v \geq  0$ and $\|v\|=1$ we get
\begin{align*}
&\int_\Om \left(\int_\Om \frac{F(y,tv)}{|y|^\alpha|x-y|^{\mu}}dy\right)\frac{F(x,tv)}{|x|^{\alpha}}~dx \geq \int_\Om \int_\Om \frac{(c_1(tv)^{K_1}(y)-c_2)(c_1(tv)^{K_1}(x)-c_2)}{|y|^\alpha |x|^\alpha |x-y|^{\mu}}~dxdy\\
  & = c_1^2 t^{2K_1} \int_\Om \int_\Om \frac{v^{K_1}(y)v^{K_1}(x)}{|y|^\alpha |x|^\alpha |x-y|^\mu}~dxdy  -2c_1c_2t^{K_1}\int_\Om \int_\Om\frac{v^{K_1}(y)}{|y|^\alpha |x|^\alpha |x-y|^\mu}~dxdy \\
  &\quad+ c_2^2 \int_\Om \int_\Om \frac{1}{|y|^\alpha |x|^\alpha  |x-y|^{\mu}}~dxdy.
\end{align*}
Then using above estimates in \eqref{functional} for $k \neq -1$, we obtain
\begin{align*}
\mathcal{J}(tv) &{\leq c_3+ c_4 t^{2} + c_5 t^{2(k+1)} - c_4t^{2K_1}+c_6t^{K_1}}
\end{align*}
and for $k=-1$
$$\mathcal{J}(tv) \leq c_3+ c_4 t^{2} + c_5 \ln(t^{2}) - c_4t^{2K_1}+c_6t^{K_1}$$
where $ c_i's$ are positive constants for $i=3, \dots,6$.
Now by choosing $t$ large enough, we obtain that there exists a $v\in W_0^{m ,2}(\Omega)$ with $\|v\|> R_0$ such that $\mathcal{J}(v)<0$.\hfill{\QED}
\end{proof}

\begin{Lemma}\label{kc-PS-bdd}
Every Palais Smale sequence of $\mathcal J$ is bounded in $W_0^{m ,2}(\Omega)$.
\end{Lemma}
\begin{proof}
Let $\{u_k\} \subset W_0^{m ,2}(\Omega)$ be a Palais Smale sequence for $\mathcal{J}$ at level $c$ (denoted by $(PS)_c$ for some $c \in \mathbb{R}$)  {\it i.e.}
\begin{equation*}
\mathcal{J}(u_k) \to c \; \text{and}\; \mathcal{J}^\prime(u_k) \to 0\;\text{as}\; k \to \infty.
\end{equation*}
Then from \eqref{functional} and \eqref{weakfunctional}, we obtain
\begin{equation}\label{kc-PS-bdd1}
\begin{split}
&\frac{1}{2} \mathcal{M}(\|u_k\|^2) - \frac12 \int_\Om \left(\int_\Om \frac{F(y,u_k)}{|y|^\alpha|x-y|^{\mu}}dy \right)\frac{F(x,u_k)}{|x|^{\alpha}}~dx \to c \; \text{as}\; k \to \infty,\\
&\left| M(\|u_k\|^{2})\int_\Om  \nabla^m u_k .\nabla^m\phi -\int_\Om \left(\int_\Om \frac{F(y,u_k)}{|y|^\alpha |x-y|^{\mu}}dy \right)\frac{f(x,u_k)}{|x|^\alpha}\phi ~dx \right|\leq \e_k\|\phi\|
\end{split}
\end{equation}
for any $\phi \in W_0^{m ,2}(\Omega)$, where $\e_k \to 0$ as $k\to \infty$. By substituting $\phi=u_k$ we get
\begin{equation}\label{kc-PS-bdd2}
\left| M(\|u_k\|^{2})\int_\Om |\nabla^m u_k|^{2}-\int_\Om \left(\int_\Om \frac{F(y,u_k)}{|y|^\alpha |x-y|^{\mu}}dy \right)\frac{f(x,u_k)u_k}{|x|^\alpha} ~dx \right|\leq \e_k\|u_k\|.
\end{equation}
Using assumption (h3), we get that there exists a $\theta>2 $ such  that $\theta F(x,t)\leq tf(x,t)$ for any $t>0$ and $ x\in \Om$ which implies
\begin{equation}\label{kc-PS-bdd3}
\theta \int_\Om \left(\int_\Om \frac{F(y,u_k)}{|y|^\alpha |x-y|^{\mu}}dy \right)\frac{F(x,u_k)}{|x|^{\alpha}}~dx \leq \int_\Om \left(\int_\Om \frac{F(y,u_k)}{|y|^\alpha |x-y|^{\mu}}dy \right)\frac{f(x,u_k)u_k}{|x|^{\alpha}} ~dx.
\end{equation}
%We notice that because of \eqref{kc-1} and Remark \ref{rem2.2} we get
%\[\int_\Om \int_\Om \frac{F(y,u_k)}{|x-y|^{\mu}}~dxdy \leq \left(\int_\Om |F(y,u_k)|^{q^\prime}dy\right)^{\frac{1}{q^\prime}} \int_\Om \left(\int_\Om \frac{1}{|x-y|^{\mu q}}dy\right)^{\frac{1}{q}}~dx \leq C_0(\|u_k\|^{\beta_0+1} + \|u_k\|^p )  \]
%where $q\in [1,2n/\mu)$.}
Now using  \eqref{kc-PS-bdd1}, \eqref{kc-PS-bdd2}, \eqref{kc-PS-bdd3} and $(m3)^\prime$, we get
\begin{equation}\label{kc-PS-bdd4}
\begin{split}
& \mathcal{J}(u_k)- \frac{1}{2\theta}\langle \mathcal{J}^\prime(u_k),u_k\rangle
=\frac{1}{2} \mathcal{M}(\|u_k\|^2)- { \frac{1}{2 \theta}} M(\|u_k\|^2)\|u_k\|^{2}\\
& \quad -\frac12 \left( \int_\Om \left(\int_\Om \frac{F(y,u_k)}{|y|^\alpha |x-y|^{\mu}}dy \right)\frac{F(x,u_k)}{|x|^\alpha}~dx+ \frac{1}{2\theta}\int_\Om \left(\int_\Om \frac{F(y,u_k)}{|y|^\alpha |x-y|^{\mu}}dy \right)\frac{f(x,u_k)u_k}{|x|^{\alpha}} ~dx\right)\\
&\geq \frac{1}{2}\mathcal{M}(\|u_k\|^{2})- \frac{1}{2 \theta} M(\|u_k\|^{2}) \|u_k\|^{2}  \geq  \left(\frac{1}{2}- \frac{1}{2 \theta}\right) M_0 \|u_k\|^2.
\end{split}
\end{equation}
Also \eqref{kc-PS-bdd1} and \eqref{kc-PS-bdd2} yields
\begin{equation}\label{kc-PS-bdd5}
\mathcal{J}(u_k)- \frac{1}{2\theta}\langle \mathcal{J}^\prime(u_k),u_k\rangle \leq C \left( 1+ \e_k \frac{\|u_k\|}{2\theta}\right)
\end{equation}
for some $C>0$. Therefore \eqref{kc-PS-bdd4} and \eqref{kc-PS-bdd5} gives us the desired result.
\hfill{\QED}
\end{proof}
\subsection{Mountain pass critical level}
To obtain bound for the mountain pass critical level in this section, we use Adams functions to construct a sequence of test functions. Let $\mathcal{B}$ denotes the unit ball and $ \mathcal{B}_l$ is the ball with center $0$ and radius $l$ in $\mathbb{R}^n$. Without loss of generality, we can assume that $\mathcal{B}_l \subset \Omega$, then from \cite[Lemma 5, p. 895]{LO}, we have the following result- For $l\in (0,1)$, there exists
\begin{equation}\label{ul}
U_l \in \{u \in W_0^{m,2}(\Omega): u|_{\mathcal{B}_l}=1 \}
\end{equation} such that $$\|U_l\|^{2}= C_{m,2}(\mathcal{B}_l; \mathcal{B}) \leq \frac{ \zeta_{m,2m}}{n \log\left(\frac{1}{l}\right)}$$ where $C_{m,2}(K,E)$ is the conductor capacity of $K$ in $E$ whenever $E$ is an open set and $K$ is relatively compact subset of $E$ and $C_{m ,2}(K;E) \eqdef \inf\{\|u\|^{2}: u \in C_0^\infty(E), u|_{K}=1\}.$ \\
Let $\tilde{x} \in \Omega$ and $R \leq  R_0= \mbox{dist}(\tilde{x}, \partial \Omega)$. Then the Adams function $\tilde A_r$ is defined as
\begin{equation*}
\tilde A_r(x)=\left\{
\begin{split}
&\left(\frac{n \log\left(\frac{R}{r}\right)}{\zeta_{m,2m}}\right)^\frac{1}{2} U_{\frac{r}{R}}\left(\frac{x-\tilde{x}}{R}\right) \  \text{if} \ |x-\tilde{x}| < R,\\
& 0  \quad \quad \quad \quad \quad \quad \quad \quad \quad \quad \quad \quad \  \text{if} \ |x-\tilde{x}| \geq  R\\
\end{split}
\right.
\end{equation*}
where $0<r<R$, ${U_{l=\frac{r}{R}}}$ is as in \eqref{ul} and ${\|\tilde{A}}_r\|\leq 1.$\\
Let $\sigma>0$ (to be chosen later), $\tilde{x}=0$, $R= \sigma$ and $r = \frac{\sigma}{k}$ for $k \in \mb N$, then we define
\begin{equation*}
{A_k(x) \eqdef \tilde{A}_{\frac{\sigma}{k}}(x)}=\left\{
\begin{split}
&\left(\frac{n \log(k)}{\zeta_{m,2m}}\right)^\frac{1}{2} U_{\frac{1}{k}}\left(\frac{x}{\sigma}\right) \  \quad \text{if} \ |x| < \sigma,\\
& 0 \ \ \ \ \quad \quad \quad \ \ \ \ \  \ \ \ \ \ \ \ \ \ \ \ \ \  \text{if} \ |x| \geq  \sigma.\\
\end{split}
\right.
\end{equation*}
Then $A_k(0)=\left(\frac{n \log(k)}{ \zeta_{m,2m}}\right)^\frac{1}{2}$ and $\|A_k\| \leq 1.$\\
 We define the mountain pass critical level as\
\begin{equation}\label{new2}
l^* = \inf_{\vartheta \in \Gamma}\max_{t\in[0,1]} \mathcal{J}(\vartheta(t)).
\end{equation}
where $\Gamma = \{\vartheta \in C([0,1], W^{m, 2}_0(\Om)):\; \vartheta(0)=0, \;\mathcal{J}(\vartheta(1))<0\}$. Now we analyze the first critical level and study the convergence of Palais-Smale sequence below this level.

\begin{Theorem}\label{PS-level}
Under the assumption \eqref{h-growth},
\begin{equation}\label{crilevel}
0<l^* < \frac{1}{2} \mathcal{M} \left( \frac{2n-(2 \alpha +\mu)}{2n}{\zeta_{m,2m}}\right).
\end{equation}
\end{Theorem}
\begin{proof}
We have observed in Lemma \ref{lemma3.1} for $u\in W^{m,2}_0(\Om)\setminus \{0\}$, $\mathcal{J}(tu) \to -\infty$ as $t\to \infty$ and $l^* \leq \max_{t\in[0,1]} \mathcal{J}(tu)$ for $u \in W_0^{m, 2}(\Omega)\backslash\{0\}$ satisfying $\mathcal{J}(u)<0$.  So it is enough to prove that there exists a $k \in \mb N$ such that
\[\max_{t\in[0,\infty)} \mathcal{J}(t A_k) < \frac{1}{2} \mathcal{M} \left( \frac{2n-(2 \alpha +\mu)}{2n} \zeta_{m,2m}\right).\]
We establish the above claim by contradiction. Suppose this is not true, then for all $k \in \mb N$ there exists a $t_k>0$ such that
\begin{equation}\label{kc-PScond0}
\begin{split}
&\max_{t\in[0,\infty)} \mathcal{J}(t A_k) = \mathcal{J}(t_k A_k) \geq \frac{1}{2} \mathcal{M}\left(\frac{2n-(2 \alpha +\mu)}{2n} \zeta_{m,2m}\right)\\
& \text{and}\;  \frac{d}{dt}(\mathcal{J}(t A_k))|_{t=t_k}=0.
\end{split}
\end{equation}
From Lemma \ref{lemma3.1} and  \eqref{kc-PScond0}, we obtain $\{t_k\}$ must be a bounded sequence in $\mb R$ and
\begin{equation}\label{kc-PScond1}
{\frac{1}{2} \mathcal{M}\left( \frac{2n-(2 \alpha +\mu)}{2n}{\zeta_{m,2m}}\right) < \frac{1}{2}\mathcal{M}(t_k^2)}
\end{equation}
Then monotonicity of $\mathcal{M}$ implies that
\begin{equation}\label{kc-PScond2}
t_k^2 > \left( \frac{2n-(2 \alpha +\mu)}{2n}{\zeta_{m,2m}}\right).
\end{equation}
%From \eqref{kc-PScond2}, we get
%\begin{equation}\label{kc-PScond3}
%t_k(\log k)^{\frac{n-m}{n}} \to \infty \;\text{as}\; k \to \infty.
%\end{equation}
Consequently, by using \eqref{kc-PScond0} and choosing $\sigma, { k} $ such that $B_{\sigma/{ k}} \subset \Omega $, we obtain
\begin{equation}\label{kc-PS-cond3}
\begin{split}
M((\|t_k A_k\|)^2) \|t_k A_k\|^2 &= \int_\Om \left(\int_\Om \frac{F(y,t_k A_k)}{|y|^{{\alpha}}|x-y|^{\mu}}dy\right)\frac{f(x,t_k A_k)t_k A_k}{|x|^{\alpha}} ~dx\\
& \geq \int_{B_{\frac{\sigma}{{ k}}}}\left(\int_{B_{\frac{\sigma}{{ k}}}}\frac{F(y,t_k A_k)}{|y|^\alpha|x-y|^\mu}~dy\right)\frac{f(x,t_k A_k)t_k A_k}{|x|^{\alpha}} ~dx.
%& \geq (\beta-\e)\exp\left( a \left( \frac{t_k(\log k)^{\frac{n-1}{n}}}{\omega_{n-1}^{\frac{1}{n}}}\right)^{\frac{n}{n-1}}\right)\int_{B_{\frac{\sigma}{{ k}}}}\int_{B_{\rho/k}} \frac{~dxdy}{|x-y|^\mu}
\end{split}
\end{equation}
For a positive constant $C_{\mu, n}$ depending on $\mu$ and $n$, we obtain (see equation $(2.11)$, page. 1943, \cite{ACTY})
\[\int_{B_{\frac{\sigma}{{ k}}}}\int_{B_{\frac{\sigma}{{ k}}}} \frac{~dxdy}{|y|^\alpha|x|^{\alpha} |x-y|^\mu} \geq C_{\mu, n} \left(\frac{\sigma}{{ k}}\right)^{2n-(2\alpha+\mu)}.\]
From \eqref{h-growth}, we know that for each $\rho>0$ there exists a $s_\rho>0$ such that
\[sf(x,s)F(x,s) \geq \rho \exp\left( 2 s^2 \right),\; \text{whenever}\; s \geq s_\rho.\]
 %Using $(h3)$ we obtain for $s_0$ small enough,{\color{red}
%\begin{align}\label{assum}
%\lim_{s \to \infty} \frac{s f(x,s)F(x,s)}{\exp\left(2 s^{\frac{n}{n-m}}\right)} \geq C(s_0) \lim_{s \to \infty} \frac{s^\ell \int_{s_0}^s \exp( t^{\frac{n}{n-m}}) dt }{\exp\left( s^{\frac{n}{n-m}}\right)} = \infty.
%\end{align}}
%where $C(s_0)=\frac{h^2(x,s_0)}{s_0^{\ell-1}}.$
Using this in \eqref{kc-PS-cond3}, we obtain, for some $C>0$
\begin{equation*}
M(\|t_k A_k\|^2)t_k^2 \geq \rho \exp \left(2 |t_k A_k(0)|^2\right)C_{\mu,n}\left(\frac{\sigma}{{k}}\right)^{2n-(2 \alpha +\mu)} \geq C \ k^{\frac{2n t_k^2}{\xi_{m, 2m}}-(2n-(2 \alpha+\mu))}.
\end{equation*}
%which implies
%\begin{equation}
%\frac{M(t_k^n)t_k^n}{k^{{a t_k^{\frac{n}{n-1}}}{\omega_{n-1}^{-\frac{1}{n-1}}}-(4-\mu)}}\geq (\beta -\e)C_\mu\rho^{4-\mu}.
%\end{equation}
Now from \eqref{kc-PScond2}, it follows that taking $k$ large enough, we arrive at a contradiction. This completes the proof of the result. \hfill{\QED}
\end{proof}
%\begin{Remark}
%We can replace the constant $\kappa_{m,n}$ in \eqref{crilevel} by any constant $C< \zeta_{m,n}.$
%\end{Remark}

\begin{Lemma}\label{kc-ws}
Let $\{u_k\}\subset W_0^{m,2}(\Om)$ be a Palais Smale sequence for $\mathcal{J}$ at $c\in \mb R$ then there exists a $u_0\in W_0^{m, 2}(\Om)$ such that as $k \to \infty$ (up to a subsequence)
\[\int_\Om \left(\int_\Om \frac{F(y,u_k)}{|y|^\alpha |x-y|^{\mu}}dy\right)\frac{f(x,u_k)}{|x|^\alpha}\phi ~dx\to \int_\Om \left(\int_\Om \frac{F(y,u_0)}{|y|^\alpha|x-y|^{\mu}}dy\right)\frac{f(x,u_0)}{|x|^\alpha}\phi~dx \]
for all $\phi\in C_c^\infty(\Om)$.
\end{Lemma}
\begin{proof}
If $\{u_k\}$ is a Palais Smale sequence at $l^*$ for $\mathcal{J}$ satisfying \eqref{kc-PS-bdd1} and \eqref{kc-PS-bdd2}. %Since $\mathcal{J}(u^+)\leq \mathcal{J}(u)$ for each $u \in W_0^{m, 2}(\Om)$, so we can assume $u_k \geq 0$ for each $k \in \mb N$.
From Lemma \ref{kc-PS-bdd}, we obtain that $\{u_k\}$ is bounded in $W_0^{m, 2}(\Om)$ so there exists a $u_0 \in W_0^{m, 2}(\Om)$ such that up to a subsequence $u_k \rightharpoonup u_0$ weakly in $W_0^{m, 2}(\Om)$, strongly in $L^{q}(\Om)$ for all $q \in [1,\infty)$ and pointwise a.e. in $\Om$ as $k \to \infty$.  Let $\Om' \subset\subset \Om$ and $\varphi \in C_c^\infty(\Om)$ such that $0\leq \varphi \leq 1$ and $\varphi \equiv 1$ in $\Om' $ then by taking $\varphi$ as a test function in \eqref{kc-PS-bdd1}, we get the following estimate
\begin{equation*}
\begin{split}
&\int_{\Om^{'}}\left( \int_\Om \frac{F(y,u_k)}{|y|^\alpha |x-y|^\mu}dy\right)\frac{f(x,u_k)}{|x|^\alpha}~dx \leq \int_\Om \left( \int_\Om \frac{F(y,u_k)}{|y|^\alpha |x-y|^\mu}dy\right)\frac{f(x,u_k)\varphi}{|x|^\alpha}~dx\\
&\leq  \e_k \left\|\varphi\right\| + M(\|u_k\|^2) \int_\Om \nabla^m u_k. \nabla^m \varphi~dx \leq \e_k \|\varphi\|+ C \|u_k\| \|\varphi\|.
\end{split}
\end{equation*}
By using $\|u_k\|\leq C_0$ for all $k$, we obtain the sequence $\{w_k\}:=\left\{\left( \int_\Om\frac{F(y,u_k)}{|y|^\alpha |x-y|^\mu}dy\right)\frac{f(x,u_k)}{|x|^\alpha}\right\}$ is bounded in $L^1_{\text{loc}}(\Om)$ which implies that up to a subsequence, $w_k \to w$ in the ${weak}^*$-topology as $k \to \infty$, where $w$ denotes a Radon measure. So for any $\phi \in C_c^\infty(\Om)$ we get
\[\lim_{k \to \infty}\int_\Om \left( \int_\Om\frac{F(y,u_k)}{|y|^\alpha|x-y|^\mu}dy\right)\frac{f(x,u_k)}{|x|^\alpha}\phi ~dx = \int_\Om \phi ~dw,\; \forall \;\phi \in C_c^\infty(\Om). \]
Since $u_k$ satisfies \eqref{kc-PS-bdd1}, for any measurable set $E \subset \Om$ and $\phi \in C_c^\infty(\Om)$ such that supp $\phi\subset E$ we get that
\begin{align*}
w(E)&=\int_E \phi~ dw= \lim_{k \to \infty} \int_E\int_\Om \left( \frac{F(y,u_k)}{|y|^\alpha|x-y|^\mu}dy\right)\frac{f(x,u_k)}{|x|^\alpha}\phi ~dx \\
&= \lim_{k \to \infty} \int_\Omega\int_\Om \left( \frac{F(y,u_k)}{|y|^\alpha|x-y|^\mu}dy\right)\frac{f(x,u_k)}{|x|^\alpha}\phi ~dx = \lim_{k \to \infty} M(\|u_k\|^2)\int_\Om \nabla^m u_k . \nabla^m \phi ~dx\\
& \leq C_1\int_E \nabla^m u . \nabla^m \phi~dx
\end{align*}
where we used (m2) in the last inequality and  weak convergence of $u_k$ to $u$ in $W^{m,2}_0(\Om)$.
%\[w(E)=\int_E \phi~ dw= \lim_{k \to \infty} M(\|u_k\|)\int_E |\nabla^m u_k|^{\frac{n}{m}-2}\nabla^m u_k \nabla^m \phi ~dx.  \]
This implies that $w$ is absolutely continuous with respect to the Lebesgue measure. Thus, Radon-Nikodym theorem establishes that there exists a function $g \in L^1_{\text{loc}}(\Om)$ such that for any $\phi\in C^\infty_c(\Omega)$, $\int_\Om \phi~ dw= \int_\Om \phi g~dx$.
Therefore for any $\phi\in C^\infty_c(\Omega)$ we get
\[\lim_{k \to \infty}\int_\Om\left( \int_\Om \frac{F(y,u_k)}{|y|^\alpha|x-y|^\mu}dy\right)\frac{f(x,u_k)}{|x|^\alpha}\phi~ ~dx = \int_\Om \phi g~dx = \int_\Om  \left( \int_\Om \frac{F(y,u_0)}{|y|^\alpha|x-y|^\mu}dy\right)\frac{f(x,u_0)}{|x|^\alpha}\phi~ ~dx \]
which completes the proof.\hfill{\QED}
\end{proof}
%In order to show that weak limit of any $(PS)_c$ sequence is a weak solution of $(KC)$, we need the following convergence lemma.
%and the proof follows similarly as Lemma $3.6$ in \cite{AGMS}.

\begin{Lemma}\label{PS-ws}
Let $\{u_k\}\subset W_0^{m, 2}(\Om)$ be a Palais Smale sequence of $\mathcal{J}$ at $c \in \mb R$ and $(h4)$ holds. Then there exists a $u \in W_0^{m, 2}(\Om)$ such that, up to a subsequence, $u_k \rightharpoonup u$ weakly in $W_0^{m, 2}(\Om)$ and
\begin{equation}\label{PS-wk0}
\left(\int_\Om \frac{F(y,u_k)}{|y|^\alpha|x-y|^{\mu}}dy\right)\frac{F(x,u_k)}{|x|^\alpha}  \to \left(\int_\Om \frac{F(y,u)}{|y|^\alpha |x-y|^{\mu}}dy\right)\frac{F(x,u)}{|x|^\alpha} \; \text{in}\; L^1(\Om)
\end{equation}
as $k \to \infty$. %Moreover, $u$ forms a weak solution of $(KC)$.
\end{Lemma}
\begin{proof}
Let $\{u_k\}\subset W_0^{m, 2}(\Om)$ be a Palais Smale sequence of $\mathcal{J}$ at level $c$ then
 from Lemma \ref{kc-PS-bdd} we know that $\{u_k\}$  must be bounded in $W_0^{m, 2}(\Om)$. Thus there exists a $u\in W_0^{m, 2}(\Om)$ such that $u_k \rightharpoonup u$ weakly in $W_0^{m, 2}(\Om)$, $u_k \to u$ pointwise a.e. in {$\Om$} and $u_k \to u$ strongly in $L^q(\Om)$, for each $q \in [1,\infty)$ as $k \to \infty$. Also from \eqref{kc-PS-bdd1}, \eqref{kc-PS-bdd2} and \eqref{kc-PS-bdd3} we get that there exists a constant $C>0$ such that
\begin{equation}\label{3.17new}
\int_\Om \left(\int_\Om \frac{F(y,u_k)}{{{|y|^{\alpha}}}|x-y|^{\mu}}~dy\right)\frac{F(x,u_k)}{{{|x|^{\alpha}}}}~dx \leq C\;\;\text{and}\;\;
\int_\Om  \left(\int_\Om \frac{F(y,u_k)}{{{|y|^{\alpha}}}|x-y|^{\mu}}~dx\right)\frac{f(x,u_k)u_k}{{{|x|^{\alpha}}}}  \leq C.
\end{equation}
Consider
\begin{align*}
&\left|\int_\Om \left(\int_\Om \frac{F(y,u_k)}{{{|y|^{\alpha}}}|x-y|^{\mu}}~dy\right)\frac{F(x,u_k)}{{{|x|^{\alpha}}}}~dx- \int_\Om \left(\int_\Om\frac{F(y,u)}{{{|y|^{\alpha}}}|x-y|^{\mu}}~dy\right)\frac{F(x,u)}{{{|x|^{\alpha}}}}~dx  \right|\\
& \quad  \quad \leq \left|\int_\Om \left(\int_\Om\frac{F(y,u_k)-F(y,u)}{{{|y|^{\alpha}}}|x-y|^{\mu}}~dy\right)\frac{F(x,u_k)}{{{|y|^{\alpha}}}}~dx\right|\\
& \quad \quad \quad \quad  + \left| \int_\Om \left(\int_\Om\frac{F(y,u)}{{{|y|^{\alpha}}}|x-y|^{\mu}}~dy\right)\frac{F(x,u_k)-F(x,u)}{{{|x|^{\alpha}}}}~dx\right| \eqdef I_1 + I_2.
\end{align*}
Using the semigroup property of the Riesz potential we can write
\begin{align}\label{semi1}
I_1 &\leq \left(\int_\Om \left(\int_\Om\frac{F(y,u_k)-F(y,u)}{{{|y|^{\alpha}}}|x-y|^{\mu}}~dy\right)\frac{F(x,u_k)-F(x,u)}{{{|x|^{\alpha}}}}~dx\right)^\frac12 \nonumber \\
& \quad \quad \quad \quad \quad \quad \quad \quad \quad \quad \quad \quad \quad \quad \quad \times \left(\int_\Om \left(\int_\Om \frac{F(y,u_k)}{{{|y|^{\alpha}}}|x-y|^{\mu}}~dy\right)\frac{F(x,u_k)}{{{|x|^{\alpha}}}}~dx \right)^\frac12.
\end{align}
\begin{align}\label{semi2}
I_2 &\leq \left(\int_\Om \left(\int_\Om\frac{F(y,u_k)-F(y,u)}{{{|y|^{\alpha}}}|x-y|^{\mu}}~dy\right)\frac{F(x,u_k)-F(x,u)}{{{|x|^{\alpha}}}}~dx\right)^\frac12 \nonumber\\
& \quad \quad \quad \quad \quad \quad \quad \quad \quad \quad \quad \quad \quad \quad \times \left(\int_\Om \left(\int_\Om \frac{F(y,u)}{{{|y|^{\alpha}}}|x-y|^{\mu}}~dy\right)\frac{F(x,u)}{{{|y|^{\alpha}}}}~dx \right)^\frac12
\end{align}
Therefore, by using \eqref{semi1} and \eqref{semi2} we obtain,
\begin{align*}
I_1 +I_2 \leq  2C \left( \int_\Om\left(\int_\Om\frac{F(y,u_k)-F(y,u)}{{{|y|^{\alpha}}}|x-y|^{\mu}}~dy\right)\frac{F(x,u_k)-F(x,u)}{{{|x|^{\alpha}}}}~dx\right)^\frac12
\end{align*}
where we used \eqref{3.17new} to get the last inequality.
Now the proof of \eqref{PS-wk0} follows similarly as the proof of equation $(3.19)$ of Lemma $3.4$ in \cite{AGMS}). %Hence from this we get $u$ forms a weak solution of $(KC)$ using Lemma \ref{kc-ws} and $u_k\rightharpoonup u$ weakly in $W^{m,2}_0(\Om)$.
\hfill{\QED}
\end{proof}

\noi Now we define the associated Nehari manifold as
 \[\mc N = \{u \in W_0^{m, 2}(\Om)\setminus \{0\}:\; \langle \mathcal{J}^\prime(u),u \rangle=0\}\]
and $l^{**} = \inf_{u\in \mc N} \mathcal{J}(u)$.
\begin{Lemma}\label{comp-lem}
If $(m3)$ and $(h3)$ holds then $l^{*} \leq l^{**}$.
\end{Lemma}
\begin{proof}
For $u \in \mc N$, we define a map $h:(0,+\infty)\to \mb R$ such that $h(t)=\mathcal{J}(tu)$. Then
\[h^\prime(t)= M(\|tu\|^2)\|u\|^2 t - \int_\Om \left(\int_\Om \frac{F(y,tu)}{|y|^\alpha |x-y|^{\mu}}dy\right)\frac{f(x,tu)u}{|x|^\alpha}~dx.  \]
and since $u \in \mc N$, therefore
\begin{align*}
h^\prime(t) &= \|u\|^{4} t^3 \left( \frac{M(\|tu\|^2)}{t^2 \|u\|^2}- \frac{M(\|u\|^2)}{\|u\|^2}\right)\\
&\quad + t^3 \left[\int_\Om\int_\Om\left(  \frac{\frac{ F(y,u)f(x,u)}{u(x)}}{|y|^\alpha |x-y|^{\mu} }dy - \int_\Om \frac{\frac{F(y,tu)f(x,tu)}{t^3 u(x)}}{|y|^\alpha |x-y|^{\mu}}dy\right)\frac{u^{2}(x)}{|x|^\alpha}~dx\right].
\end{align*}
From (h3), we get
\begin{align*}
t_1f(x,t_1)- 2F(x,t_1) \leq t_1f(x,t_1)-2 F(x,t_2)+2 \frac{f(x,t_2)}{t_2}(t_2^2-t_1^2)\leq t_2f(x,t_2)-2 F(x,t_2).
\end{align*}
for $0<t_1<t_2.$ Using this we get that $tf(x,t)-2 F(x,t)\geq 0$ for $t\geq 0$ and for any $x \in \Omega$, $t\mapsto tf(x,t)-2 F(x,t)$ is increasing on $\R^+$, which further implies that $t \mapsto \frac{F(x,tu)}{t^{2}}$ is non-decreasing for $t>0$. Therefore for $0<t<1$ and $x \in \Om$, we get $\frac{F(x,tu)}{t^2} \leq F(x,u)$ and (h3) gives that $\frac{f(x,u)}u \geq \frac{f(x,tu)}{tu} $ then
\begin{align*}
h^\prime(t) &\geq \|u\|^{4}t^{3}\left( \frac{M(\|tu\|^2)}{\|tu\|^2}- \frac{M(\|u\|^2)}{\|u\|^2}\right)\\
&\quad+ t^{3}\left[\int_\Om \left(\int_\Om\left(F(y,u)- \frac{F(y,tu)}{t^2}\right)~\frac{dy}{{|y|^\alpha}|x-y|^{\mu}}\right){\frac{f(x,tu)u^{2}(x)}{|x|^\alpha tu(x)}} ~dx\right].
\end{align*}
This gives that $h^\prime(t)\geq 0$ for $0<t\leq1$ and similarly we can show that $h^\prime(t)<0$ for $t>1$. Hence $\mathcal{J}(u)= \max_{t\geq 0} \mathcal{J}(tu)$. Now we define $g:[0,1] \to W_0^{m, 2}(\Om)$ as $g(t)=(t_0u)t$ where $t_0>1$ is such that $\mathcal{J}(t_0u)<0$. So $g \in \Gamma$, where $\Gamma$ is as defined in the definition of $l^*$. Therefore,
\[l^* \leq \max_{t\in[0,1]}\mathcal{J}(g(t)) \leq \max_{t\geq 0} \mathcal{J}(tu)=\mathcal{J}(u). \]
and since $u \in \mc N$ is arbitrary, so we get $l^* \leq l^{**}$.\hfill{\QED}
\end{proof}

\noi Now we recall the following higher integrability Lemma from \cite{L}( also refer Lemma $1$ of \cite{LO}).
\begin{Lemma}\label{Lions-lem}
Let $\{v_k \} $ be a sequence in $W_0^{m, 2}(\Om)$  such that $\|v_k\|=1$ converging weakly to a non zero $v \in W_0^{m, 2}(\Om)$. Then for every $p< \frac{1}{1-\|v\|^2}$,
\[\sup_k \int_\Om {\exp \left( p \zeta_{m,2m} |v_k|^{2}\right)}<+\infty. \]
\end{Lemma}
\vspace{.2cm}
\noi \textbf{Proof of Theorem \ref{kc-mt-1}:} Let $\{u_k\}$ be a $(PS)_{l^*}$ sequence at the critical level $l^*$ and hence considered as a minimizing sequence associated to the variational problem \eqref{new2}. Then by Lemma \ref{PS-ws}, there exists a $u_0 \in W_0^{m, 2}(\Om)$ such that up to a subsequence
$u_k \rightharpoonup u_0$ weakly in $W_0^{m, 2}(\Om)$ as $k \to \infty$. % We know that $u_k \rightharpoonup u_0$ weakly in $W_0^{m, 2}(\Om)$ and
%\[\int_\Om \left(\int_\Om \frac{ F(y,u_k)}{|x-y|^{\mu}}dy\right)f(x,u_k)u_k ~dx \leq C \int_\Om f(x,u_k)u_k ~dx.  \]
First we claim that $u_0$ is non-trivial. So if $u_0 \equiv 0$ then using Lemma \ref{PS-ws}, we infer that
\[\int_\Om \left(\int_\Om \frac{ F(y,u_k)}{|y|^\alpha |x-y|^{\mu}}dy\right)\frac{F(x,u_k)}{|x|^\alpha} ~dx  \to 0\; \text{as}\; k \to \infty.\]
Therefore $ \lim_{k \to \infty} \mathcal{J}(u_k) = \frac{1}{2}\lim_{k\to\infty} \mathcal{M}(\|u_k\|^{2}) = l^*$ and then for large enough $k$ Theorem \ref{PS-level} gives
\[\mathcal{M}(\|u_k\|^2) < \mathcal{M} \left( \frac{2n-(2 \alpha +\mu)}{2n}{\zeta_{m,2m}}\right).\]
Then by monotonicity of {$\mathcal{M}$}, we obtain
\[\frac{2n}{2n-(2 \alpha +\mu)}\|u_k\|^{2} < \zeta_{m,2m}.  \]
Now, this implies that we can choose a  $q >\frac{2n}{2n-(2 \alpha +\mu)}$ such that $\sup_{k} \int_{\Omega} |{f(x,u_k)}|^q ~dx < +\infty$. Using  Proposition \ref{HLS}, Theorem \ref{TM-ineq} and the Vitali's convergence theorem we conclude that
\[\int_{\Om}\left( \int_\Om\frac{F(y,u_k)}{|y|^\alpha |x-y|^{\mu}}dy\right)\frac{f(x,u_k)u_k}{|x|^\alpha}~dx \to 0 \;\text{as}\; k \to \infty.\]
Hence $\lim_{k\to \infty}\langle \mathcal{J}^\prime(u_k),u_k \rangle=0$ which gives $\lim_{k\to \infty}M(\|u_k\|^2)\|u_k\|^2=0$. From (m1) we then obtain $\lim_{k \to \infty} \|u_k\|^2=0$. Thus using Lemma \ref{PS-ws}, it must be that $\lim_{k \to \infty} \mathcal{J}(u_k)=0 =l^*$ which contradicts $l^*>0$. Thus $u_0 \not \equiv 0$.
Now we show that $u_0\geq 0$ in $\Om$. From Lemma \ref{kc-PS-bdd} we know that $\{u_k\}$ must be bounded. Therefore there exists a constant $\rho >0$ such that up to a subsequence $\|u_k\| \to \rho$ as $k \to \infty$.  Let $ \varphi \in W_0^{m, 2}(\Om)$ then  by Lemma~\ref{kc-ws} we have
 $$\int_\Om \left(\int_\Om \frac{F(y,u_k)}{|y|^\alpha |x-y|^{\mu}}dy\right)\frac{f(x,u_k)}{|x|^\alpha} \varphi ~dx \to \int_\Om \left(\int_\Om \frac{F(y,u_0)}{|y|^\alpha|x-y|^{\mu}}dy\right)\frac{f(x,u_0)}{|x|^\alpha}\varphi~dx\;\text{as}\; k \to \infty.$$
 Since $\mathcal{J}^\prime(u_k) \to 0$ and $u_k \rightharpoonup  u_0$ weakly in $W^{m,2}_0(\Om)$, we get
 \begin{align*}
 M(\rho^2)\int_\Om \nabla^m u_0. \nabla^m \varphi~dx = \int_\Om \left(\int_\Om \frac{F(y,u_0)}{|y|^\alpha |x-y|^{\mu}}dy\right)\frac{f(x,u_0)}{|x|^\alpha}\varphi~dx,
 \end{align*}
 as $k \to \infty$. In particular, taking $\varphi = u_0^-$ in the above equation we get $M(\rho^2)\|u_0^-\|^2=0$ which implies together with assumption (m1) that $u_0^-=0$ a.e. in $\Om$. Therefore $u_0\geq 0$ a.e. in $\Om$.
%From growth assumptions on $f$, we have $\frac{f(\cdot,u_0)}{|x|^\alpha} \in L^q(\Om)$ for $1\leq q <\infty$. From \eqref{kc-1} we know that $F(y,u_0) \in L^r(\Om)$ for all $ r \in [1,\infty)$. Since $ \mu \in (0,n)$ and $0<\alpha < n-\mu$, we get $ |y|^{-\alpha} \in L^{r_0}(\Omega), \;|x-y|^{-\mu}  \in L^{r_1}(\Omega)$ for some $r_0 \in (1, \frac{n}{\alpha}), r_1 \in (1, \frac{n}{\mu}).$ So by using H\"older inequality we obtain
%\begin{align}\label{infty}
%\int_\Om \frac{F(y,u_0)}{|y|^\alpha |x-y|^{\mu}}dy \in L^\infty(\Omega)
%\end{align}
% Hence $\left(\int_\Om \frac{F(y,u_0)}{|y|^\alpha |x-y|^{\mu}}~dy \right)\frac{f(x,u_0)}{|x|^\alpha} \in L^q(\Om)$ for { $1 \leq q <\infty$}. By elliptic regularity results, we finally get that $u_0 \in L^\infty(\Om)$ and $u_0 \in C^{1,\gamma}(\overline{\Om})$ for some $\gamma \in (0,1)$.  Therefore, $u_0>0$ in $\Om$ follows from the strong maximum principle and $u_0\not\equiv 0$.
 \\
\textbf{Claim (1):} $M(\|u_0\|^2)\|u_0\|^2 \geq \int_\Om \left(\int_\Om \frac{F(y,u_0)}{|y|^\alpha |x-y|^{\mu}}~dy \right)\frac{f(x,u_0)u_0}{|x|^\alpha} ~dx.$\\
Arguing by contradiction, suppose that
\[ M(\|u_0\|^2)\|u_0\|^2 <\int_\Om \left(\int_\Om \frac{F(y,u_0)}{|y|^\alpha |x-y|^{\mu}}~dy \right)\frac{f(x,u_0)u_0}{|x|^\alpha}~dx\]
which implies that $\langle \mathcal{J}^\prime(u_0),u_0 \rangle <0$. For $t>0$, using the map $t \mapsto tf(x,t)-2 F(x,t)$ is increasing on $\R^+$, we have
\begin{align*}
\langle \mathcal{J}^\prime(t u_0),u_0 \rangle  &\geq M(\| tu_0\|^2) { t \|u_0\|^{2}} -  \frac{1}{2}\int_\Om  \left(\int_\Om \frac{f(y,tu_0)tu_0(y)}{|y|^\alpha|x-y|^{\mu}}~dy \right)\frac{f(x,tu_0)u_0}{|x|^\alpha}~dx\\
 &\geq M_0 t \|u_0\|^2 - \frac{1}{2}\int_\Om \left(\int_\Om \frac{f(y,tu_0)tu_0(y)}{|y|^\alpha |x-y|^{\mu}}~dy \right)\frac{f(x,tu_0)u_0}{|x|^\alpha}~dx.
\end{align*}
Since (h3) gives that
\[\lim_{t\to 0^+} \frac{f(x,t)}{t^\gamma}=0 \; \text{uniformly in}\; x\in \Om,\; \text{for all}\; \gamma \in [0,1],\]
we can choose $t>0$ sufficiently small so that $\langle \mathcal{J}^\prime(tu_0),u_0 \rangle>0$. Thus there exists a $t_*\in (0,1)$ such that $\langle \mathcal{J}^\prime(t_*u_0),u_0 \rangle=0$ {\it i.e.} $t_*u_0 \in \mc N$. So using Lemma \ref{comp-lem} and $(m3)^\prime$ we get
\begin{align*}
l^* &\leq l^{**} \leq \mathcal{J}(t_*u_0) = \mathcal{J}(t_*u_0) - \frac{1}{4}\langle \mathcal{J}^\prime(t_*u_0),t_*u_0 \rangle\\
& = \frac{\mathcal{M}(\|t_*u_0\|^2)}{2} -\frac12 \int_\Om \left(\int_\Om \frac{ F(y,t_*u_0)}{|y|^\alpha |x-y|^{\mu}}dy\right)\frac{F(x,t_*u_0)}{|x|^\alpha}~dx \\
&\quad-\frac{1}{4}M(\|t_*u_0\|^2)\|t_*u_0\|^2 + \frac{1}{4}\int_\Om \left(\int_\Om \frac{F(y,t_*u_0)}{|y|^\alpha |x-y|^{\mu}}dy\right)\frac{f(x,t_*u_0)t_*u_0}{|x|^\alpha}~dx\\
%& < \frac{\mathcal{M}(\|u_0\|^2)}{2}  - \frac{1}{4} M(\|u_0\|^2)\|u_0\|^2 \\
%& \quad + \frac{1}{4}\int_\Om \left(\int_\Om \frac{ F(y,t_*u_0)}{|y|^\alpha |x-y|^{\mu}}dy\right)\frac{f(x,t_*u_0)t_*u_0 - 2F(x,t_*u_0)}{|x|^\alpha}~dx\\
& <  \frac{\mathcal{M}(\|u_0\|^2)}{2}  -\frac{1}{4} M(\|u_0\|^2)\|u_0\|^2 \\
&\quad + \frac{1}{4}\int_\Om \left(\int_\Om \frac{ F(y,u_0)}{|y^{\alpha}||x-y|^{\mu}}dy\right)\frac{f(x,u_0)u_0 -2F(x,u_0)}{|x|^\alpha}~dx\\
%& \leq \liminf_{k \to \infty} \left[\frac{\mathcal{M}(\|u_k\|^2)}{2}  -\frac{1}{4} M(\|u_k\|^2)\|u_k\|^2\right. \\
%& \left.+ \frac{1}{4}\int_\Om \left(\int_\Om \frac{ F(y,u_k)}{|y|^\alpha |x-y|^{\mu}}dy\right)\frac{f(x,u_k)u_k -2F(x,u_k)}{|x|^\alpha}~dx\right]\\
& = \liminf_{k \to \infty} \left( \mathcal{J}(u_k) - \frac{1}{4}\langle \mathcal{J}^\prime(u_k),u_k \rangle\right) = l^*.
\end{align*}
This gives a contradiction and hence {Claim (1)} holds. \\
\textbf{Claim (2):} $\mathcal{J}(u_0)= l^*$.\\
From Lemma \ref{PS-ws} we know that
\[\int_\Om \left(\int_\Om \frac{ F(y,u_k)}{|y|^\alpha |x-y|^{\mu}}dy\right)\frac{F(x,u_k)}{|x|^\alpha}~dx \to \int_\Om \left(\int_\Om \frac{ F(y,u_0)}{|y|^\alpha |x-y|^{\mu}}dy\right)\frac{F(x,u_0)}{|x|^\alpha}~dx \]
and by using the weakly lower semicontinuity of norms in $\lim_{k \to \infty}\mathcal{J}(u_k)= l^*$, we obtain $\mathcal{J}(u_0) \leq l^*$. If $\mathcal{J}(u_0)< l^*$ then it must be $\lim_{k \to \infty} \mathcal{M}(\|u_k\|^2) > \mathcal{M}(\|u_0\|^2)$ which implies that $\lim_{k \to \infty}\|u_k\|^2 > \|u_0\|^2$, since $\mathcal{M}$ is continuous and increasing. From this we get $\rho^2 > \|u_0\|^2.$ Moreover we have
\begin{equation}\label{kc-mt-12}
\mathcal{M}(\rho^2) =  \left( 2l^* + \int_\Om \left(\int_\Om \frac{ F(y,u_0)}{|y|^\alpha|x-y|^{\mu}}dy\right)\frac{F(x,u_0)}{|x|^\alpha}~dx\right).
\end{equation}
Now we define the sequence $v_k = \frac{u_k}{\|u_k\|}$ and $v_0 = \frac{u_0}{\rho}$ such that $v_k \rightharpoonup v_0$ weakly in $W_0^{m,2}(\Om)$ and $ \|v_0\|<1$. Then from Lemma \ref{Lions-lem} we obtain
\begin{equation}\label{kc-mt-13}
\sup_{ k \in \mb N} \int_\Om {\exp \left( p|v_k|^2  \right)}<+\infty,\; \text{for} \; p<\frac{\zeta_{m,2m}}{(1-\|v_0\|^2)}.
\end{equation}
Also from $(m3)^\prime$, Claim (1) and proof of Lemma \ref{comp-lem} we obtain
\begin{align*}
\mathcal{J}(u_0)& = \frac{1}{2}\mathcal{M}(\|u_0\|^2)- \frac{1}{4} M(\|u_0\|^2)\|u_0\|^2\\
&\quad+ \frac{1}{4}\int_\Om \left(\int_\Om \frac{ F(y,u_0)}{|y|^\alpha|x-y|^{\mu}}dy\right)\frac{(f(x,u_0)u_0-2 F(x,u_0))}{|x|^\alpha}~dx\geq 0.
\end{align*}
Using this with \eqref{kc-mt-12} and Theorem \ref{PS-level} we get that
\begin{align*}
\mathcal{M}(\rho^2) = 2 l^* - 2 \mathcal{J}(u_0) +\mathcal{M}(\|u_0\|^2) < \mathcal{M} \left( \frac{2n-(2 \alpha +\mu)}{2n}\zeta_{m,2m}\right) + \mathcal{M}(\|u_0\|^2)
\end{align*}
which implies together with (m1) that
\begin{equation*}
\rho^2 < \frac{\zeta_{m,2m}\left( \frac{2n-(2 \alpha +\mu)}{2n}\right)}{1-\|v_0\|^2}.
\end{equation*}
Thus it is possible to find a $\rho_* > 0$ such that for $k\in \mb N$ large enough
\[\|u_k\|^{2} < \rho_* < \frac{\zeta_{m,2m} \left(2n-(2 \alpha +\mu)\right)}{2n (1-\|v_0\|^2)}.\]
%< \frac{\alpha_n}{(1-\|v_0\|^n)^{\frac{1}{n}}}.\]
Then we choose a $q>1$ but close to $1$ such that
\[\frac{2n}{2n-(2 \alpha +\mu)} q \|u_k\|^{2} \leq \frac{2n}{2n-(2 \alpha +\mu)}\rho_* < \frac{\zeta_{m,2m}}{(1-\|v_0\|^2)}.
 \]
Therefore from \eqref{kc-mt-13} we conclude that
\begin{equation}\label{kc-mt-15}
 \int_\Om {\exp\left(\frac{2n}{2n-(2 \alpha +\mu)} q |u_k|^{2}\right)} \leq C
 \end{equation}
for some constant $C>0$. Using \eqref{kc-mt-15} and ideas similar as in Lemma \ref{PS-ws} we obtain
\[\int_\Om \left(\int_\Om \frac{ F(y,u_k)}{|y|^\alpha |x-y|^{\mu}}dy\right)\frac{f(x,u_k)u_k}{|x|^\alpha}~dx \to \int_\Om \left(\int_\Om \frac{ F(y,u_0)}{|y|^\alpha |x-y|^{\mu}}dy\right)\frac{f(x,u_0)u_0}{|x|^\alpha}~dx. \]
We conclude that $\|u_k\| \to \|u_0\|$ and we get a contradiction to the fact that $\mathcal{J}(u_0) < l^*$ . Hence $\mathcal{J}(u_0)= l^*= \lim_{k \to \infty} \mathcal{J}(u_k)$ and $\|u_k\| \to \rho$ implies $\rho= \|u_0\|$. Then finally we have,
\begin{equation*}
M(\|u_0\|^{2}) \int_\Om \nabla^m u_0. \nabla^m \varphi ~dx = \int_\Om \left(\int_{\Om}\frac{F(y,u_0)}{{ |y|^\alpha} |x-y|^{\mu}}dy\right)\frac{f(x,u_0)}{|x|^\alpha}\varphi ~dx.
\end{equation*}
for all $\varphi \in W^{m,2}_0(\Om)$ and which completes the proof of Theorem~\ref{kc-mt-1}.
 \hfill{\QED}

%%%%%%%%%%%%%%%%%%%%%%%%%%%%%%%%%%%%%%%%%%%%%%%%%

 \section{The problem $(\mathcal P_{\la,M})$}
In this section, we consider the problem $(\mathcal{P}_{\la, M})$ with Kirchhoff non-linearity of the form $M(t)=at+b$ where $a,b>0$. We observe that $\mathcal{J}_{\la, M}$ (defined in Section 1) is unbounded on $W_0^{m,2}(\Omega)$ but bounded below on suitable subsets of $W_0^{m,2}(\Omega)$. To show the existence of weak solutions to  $(\mathcal{P}_{\la, \mathcal{M}})$, we establish the existence of minimizers of $\mathcal{J}_{\la, M}$ under the natural constraint of the Nehari Manifold which contains every solution of $(\mathcal{P}_{\la, M}).$ So we define the Nehari manifold as
\begin{equation*}
N_{\la, M} := \{u \in W_0^{m,2}(\Omega)\setminus \{0\}|\  \langle \mathcal{J}^{'}_{\la,M} (u),u \rangle =0 \}
\end{equation*}
where $\langle . \ , . \rangle$ denotes the duality between $W_0^{m,2}(\Omega)$ and $W^{-m,2}(\Omega)$ {\it i.e.} $u \in N_{\la, M}$ if and only if
\begin{equation}\label{ndef}
\|u\|^2 \;  M(\|u\|^2)- \la \int_{\Omega} h(x) |u|^{q+1} ~dx - \int_{\Omega} \left(\int_{\Omega} \frac{F(u)}{|x-y|^\mu |y|^\alpha}~dy\right) \frac{f(u) u}{|x|^\alpha} ~dx =0.\end{equation}
For $u \in W_0^{m,2}(\Omega)$, we define the fiber map $\Phi_{u,M}$ introduced by Drabek and Pohozaev in \cite{DP} as $\Phi_{u,M} : \mathbb{R}^{+} \rightarrow \mathbb{R}$ such that $\Phi_{u,M} (t)= \mathcal{J}_{\la, M}(tu)$. Thus we get
%\begin{equation*}
%\Phi_{u,M} (t)= \frac{1}{2}\mathcal{M}(\|tu\|^2)  - \frac{\la}{q+1} \int_{\Omega} h(x) |tu|^{q+1} ~dx - \frac{1}{2} \int_{\Omega} \left(\int_{\Omega} \frac{F(tu)}{|x-y|^\mu |y|^\alpha}~dy\right) \frac{F(tu)}{|x|^\alpha} ~dx,
%\end{equation*}
\begin{equation*}
\Phi^{'}_{u,M} (t) = t \|u\|^2 M(\|tu\|^2) - \la t^{q} \int_{\Omega} h(x) |u|^{q+1}~dx - \int_{\Omega} \left( \int_{\Omega} \frac{F(tu)}{|x-y|^\mu |y|^\alpha} ~dy \right) \frac{f(tu) u}{|x|^\alpha} ~dx
\end{equation*}
and
\begin{align*}
      \Phi^{''}_{u,M} (t)&= 2t^2\|u\|^{4} M'(\|tu\|^2)+  \|u\|^2 M(\|tu\|
      ^2) - \la q t^{q-1} \int_{\Omega} h(x) |u|^{q+1}~dx \\
      &- \int_{\Omega} \left( \int_{\Omega} \frac{f(tu) u }{|x-y|^\mu |y|^\alpha} ~dy \right) \frac{f(tu) u}{|x|^\alpha} ~dx - \int_{\Omega} \left( \int_{\Omega} \frac{F(tu)}{|x-y|^\mu |y|^\alpha} ~dy \right) \frac{f'(tu) u^2}{|x|^\alpha} ~dx.
 \end{align*}
Since the fiber map introduced above are closely related to Nehari manifold by the relation $tu \in N_{\la, M}$ iff $\Phi_{u,M}^{'}(t)=0$, so we analyze the geometry of the energy functional on the following components of the Nehari Manifold:
\begin{equation*}
N^+_{\la, M}:= \{u \in N_{\la,M} : \Phi_{u,M}^{''}(1)>0\} =\{tu \in W_0^{m,2}(\Omega)\setminus\{0\} : \Phi_{u,M}^{'}(t)=0, \Phi_{u,M}^{''}(t)>0\},
\end{equation*}
\begin{equation*}
N^-_{\la, M}:= \{u \in N_{\la,M} : \Phi_{u,M}^{''}(1)<0\} =\{tu \in W_0^{m,2}(\Omega)\setminus\{0\} : \Phi_{u,M}^{'}(t)=0, \Phi_{u,M}^{''}(t)<0\},
\end{equation*}
\begin{equation*}
N^0_{\la, M}:= \{u \in N_{\la,M} : \Phi_{u,M}^{''}(1)=0\} =\{tu \in W_0^{m,2}(\Omega)\setminus\{0\} : \Phi_{u,M}^{'}(t)=0, \Phi_{u,M}^{''}(t)=0\}.
\end{equation*}
Due to presence of sign changing non-linearity in $(\mathcal{P}_{\la,M})$, we also decompose $W_0^{m ,2}(\Omega)
$ into the following sets to study the behavior of fibering maps $\Phi_{u,M}$. We define $H(u)=\int_{\Omega} {h(x)}|u|^{q+1}~dx$ and
\begin{equation*}
H^{+}:= \{u \in W_0^{m,2}(\Omega) : H(u) >0 \},
\end{equation*}
\begin{equation*}
H^{-}_0:= \{u \in W_0^{m,2}(\Omega) : H(u) \leq 0 \}.
\end{equation*}
%\begin{equation*}
%H_{0}:= \{u \in W_0^{m,\frac{n}{m}}(\Omega) : H(u) =0 \}.
%\end{equation*}
\subsection{Fiber Map Analysis}
In this section, we study the geometry of $\mc J_{\la,M}$ on the Nehari manifold. We split the study according to the  decomposition of $N_{\la,M}$ and the sign of $H(u)$.
%Precisely, we prove the following:
%{\color{blue}\begin{Theorem}\label{ana}
%Let $u \in W_0^{m,2}(\Omega)$ then
%\begin{enumerate}
%\item For $u \in H_0^-$, there exists a unique $t^*$ such that $t^* u \in \mathcal{N}^-_{\lambda, M}$ for every $\lambda>0.$
%\item For $u \in H^+$, there exist $\lambda_0 >0$ and unique $t_1(u)$ and $t_2(u)$ such that $t_1 u \in \mathcal{N}^-_{\lambda, M}$ and $t_2 u \in \mathcal{N}^+_{\lambda, M}$ for every $\lambda \in (0, \lambda_0).$
%\end{enumerate}
%\end{Theorem}}
\noindent Define $\psi : \mathbb{R}^+ \rightarrow \mathbb{R}$ such that
\begin{equation*}
\psi_u(t)= t^{1-q} M(\|tu\|^2) \|u\|^2 - t^{-q} \int_{\Omega} \left( \int_{\Omega} \frac{F(tu)}{|x-y|^\mu |y|^\alpha} ~dy \right) \frac{f(tu) u}{|x|^\alpha} ~dx.
\end{equation*}
\textbf{Case 1}: $ u \in H^-_0\setminus\{0\}$\\
Since
\begin{align*}
\Phi^{'}_{u,M}(t)
%&= t M(\|tu\|^2) \|u\|^2- \la t^q \int_{\Omega} h(x) |u|^{q+1} ~dx - \int_{\Omega} \left( \int_{\Omega} \frac{F(tu)}{|x-y|^\mu |y|^\alpha} ~dy \right) \frac{f(tu) u}{|x|^\alpha} ~dx\\
&= t^q (\psi_u(t) - \la \int_{\Omega} h(x) |u|^{q+1}~dx),
\end{align*}
so $tu \in N_{\la,M}$ iff $t>0$ is a solution of $ \psi_u(t)= \la \int_{\Omega} h(x) |u|^{q+1}dx.$ We have
\begin{equation}\label{KC21}
\begin{split}
\psi^{'}_u(t)&= \left(1-q\right) t^{-q} M(\|tu\|^2) \|u\|^2 + 2  t^{2-q} M{'}(\|tu\|^2) \|u\|^{4}\\
 &+ \frac{q}{t^{q+1}} \int_{\Omega} \left( \int_{\Omega} \frac{F(tu)}{|x-y|^\mu |y|^\alpha} ~dy \right) \frac{f(tu) u}{|x|^\alpha} ~dx - t^{-q} \bigg[ \int_{\Omega} \left( \int_{\Omega} \frac{f(tu) u}{|x-y|^\mu |y|^\alpha} ~dy \right) \frac{f(tu) u}{|x|^\alpha} ~dx \\
 &+ \int_{\Omega} \left( \int_{\Omega} \frac{F(tu)}{|x-y|^\mu |y|^\alpha} ~dy \right) \frac{f'(tu) u^2}{|x|^\alpha} ~dx \bigg].
\end{split}
\end{equation}
Due to the presence of exponential non-linearity, for large $t$ we have $\psi_u^{'}(t) <0$ and since $u \in H^-_0$, there exists a unique $t^*>0$ such that $\psi_u(t^*)= \la \int_{\Omega} h(x) |u|^{q+1}dx$, {\it i.e.} $t^*u \in N_{\la, M}$.\\
Suppose there exists an another point $t_1$ ($t^*< t_1$) such that $\psi_u(t_1)= \la \int_{\Omega} h(x) |u|^{q+1} \leq 0$, {\it i.e.}
 \begin{equation}\label{role}
 t_1^{1-q} (a t_1^2\|u\|^2 +b) \|u\|^2 \leq t_1^{-q} \int_{\Omega} \left( \int_{\Omega} \frac{F(t_1u)}{|x-y|^\mu |y|^\alpha} ~dy \right) \frac{f(t_1u) u}{|x|^\alpha} ~dx
 \end{equation}
 and $\psi_u'(t_1) \geq 0$. Then from \eqref{role} and by using $f'(t_1u) t_1u > (p+1) f(t_1u)$, $f(t) t \geq (p+2) F(t)$ we obtain,
    {\begin{align*}
\psi'_u(t_1)& <\left(3-q\right)\left[t_1^{-q} (a t_1^2 \|u\|^2 +b) \|u\|^2- t_1^{-q-1} \int_{\Omega} \left( \int_{\Omega} \frac{F(t_1u)}{|x-y|^\mu |y|^\alpha} ~dy \right) \frac{f(t_1u) u}{|x|^\alpha} ~dx \right] \leq 0.
  \end{align*}}
which is  a contradiction. Also for $0<t<t^*$, $\Phi^{'}_{u,M}(t)= t^q(\psi_u(t)-\la \int_{\Omega} h(x) |u|^{q+1} ~dx) > 0$. Consequently, $\Phi_{u,M}$ is increasing in $(0,t^*)$ and also decreasing on $(t^*, \infty).$ Therefore $t^*$ is unique critical point of $\Phi_{u, M}$ which is also a point of global maximum. Furthermore, since $\psi'_u(t) = \displaystyle\frac{{\left(t \Phi''_{u,M}(t)-q \Phi'_{u,M}(t)\right)}}{t^q}$, therefore $t^* u \in N^-_{\la,M}.$\\ \\
\textbf{Case 2}: $ u \in H^+$\\
In this case, we establish that there exists a $\la_0 >0$ and a $t_*>0$ such that for $\la \in (0, \la_0)$, $\Phi_{u}$ has exactly two critical points $t_1(u)$ and $t_2(u)$ such that $t_1(u) < t_*(u) <t_2(u)$ where $t_1(u)$ is local minimum point and $t_2(u)$ is local maximum point. To prove this, we require further analysis and a priori estimates.
\subsection{Preliminary Results for Case-2}
For small $t>0,$ $\psi_u(t)>0$ and $\psi_u(t) \rightarrow -\infty$ as $ t \rightarrow \infty$ for $u \in H^+$. Then there exists at least one point $t^*$ such that $\psi_u^{'}(t_*)=0$, {\it i.e.}
 \begin{align*}
 &\left(3-q\right) t_*^{2-q} a \|u\|^4 +\left(1-q\right) t_*^{-q} b \|u\|^2 + \dfrac{q}{t_*^{q+1}} \int_{\Omega} \left( \int_{\Omega} \frac{F(t_*u)}{|x-y|^\mu |y|^\alpha} ~dy \right) \frac{f(t_*u) u}{|x|^\alpha} ~dx\\
 &=t_*^{-q}\bigg[\int_{\Omega} \left( \int_{\Omega} \frac{F(t_*u)}{|x-y|^\mu |y|^\alpha} ~dy \right) \frac{f'(t_*u) u^2}{|x|^\alpha} ~dx +\int_{\Omega} \left( \int_{\Omega} \frac{f(t_*u)u }{|x-y|^\mu |y|^\alpha} ~dy \right) \frac{f(t_*u) u}{|x|^\alpha} ~dx\bigg].
\end{align*}
%This implies that
 % \begin{align*}
 %&\left(3-q\right)a \|t_*u\|^{4} +\left(1-q\right) b \|t_*u\|^2 + q \int_{\Omega} \left( \int_{\Omega} \frac{F(t_*u)}{|x-y|^\mu |y|^\alpha} ~dy \right) \frac{f(t_*u) t_* u}{|x|^\alpha} ~dx\\
 %&=\int_{\Omega} \left( \int_{\Omega} \frac{F(t_*u)}{|x-y|^\mu |y|^\alpha} ~dy \right) \frac{f'(t_*u) (t_*u)^2}{|x|^\alpha} ~dx +\int_{\Omega} \left( \int_{\Omega} \frac{f(t_*u) (t_*u)}{|x-y|^\mu |y|^\alpha} ~dy \right) \frac{f(t_*u) t_*u}{|x|^\alpha} ~dx.
%\end{align*}
So by AM-GM inequality we obtain
%\begin{align*}
%2 \sqrt{\left(3-q\right)a \|t_*u\|^{4}b \left(1-q\right)\|t_*u\|^2} \leq B(t_*u)
%\end{align*} from which it follows
$2 \sqrt{\left(3-q\right)a b \left(1-q\right)} \|t_*u\|^{3} \leq B(t_*u)$
where $$B(u)=\int_{\Omega} \left( \int_{\Omega} \frac{F(u)}{|x-y|^\mu |y|^\alpha} ~dy \right) \frac{f'(u) u^2}{|x|^\alpha} ~dx +\int_{\Omega} \left( \int_{\Omega} \frac{f(u) u }{|x-y|^\mu |y|^\alpha} ~dy \right) \frac{f(u) u}{|x|^\alpha} ~dx.$$
Using $\psi^{'}_u(t_*)=0$, we replace the value of $a\|t_*u\|^{4}$ in the definition of $\psi_u(t)$ to obtain \\
\begin{equation}\label{KC1}
\psi_u(t_*)= \frac{1}{\left(3-q\right) t_*^{q+1}} \bigg[ B(t_*u)-3 \int_{\Omega} \left( \int_{\Omega} \frac{F(t_*u)}{|x-y|^\mu |y|^\alpha} ~dy \right) \frac{f(t_*u) t_*u}{|x|^\alpha} ~dx + 2 b \|t_* u\|^2  \bigg].
\end{equation}
 \begin{Lemma}\label{inf}
 Let $$\Gamma := \left\{ u \in W_0^{m,2}(\Omega): \|u\|^{3} \leq \frac{B(u)}{2 \sqrt{\left(3-q\right)a b \left(1-q\right)}}\right\}.$$
  %where $B(u)=\int_{\Omega} \left( \int_{\Omega} \frac{F(u)}{|x-y|^\mu |y|^\alpha} ~dy \right) \frac{f'(u) u^2}{|x|^\alpha} ~dx +\int_{\Omega} \left( \int_{\Omega} \frac{f(u) u }{|x-y|^\mu |y|^\alpha} ~dy \right) \frac{f(u) u}{|x|^\alpha} ~dx.$
  Then there exists a $\la_0>0$ such that for every $ \la \in (0, \la_0)$, $\Gamma_0 >0$ holds where
\begin{equation*}\label{imp}
\Gamma_0 := \inf_{u \in \Gamma\backslash \{0\} \cap H^+}\bigg[ B(u)- 3 \int_{\Omega} \left( \int_{\Omega} \frac{F(u)}{|x-y|^\mu |y|^\alpha} ~dy \right) \frac{f(u) u}{|x|^\alpha} ~dx + 2 b \|u\|^2 - \la \left(3-q\right) H(u) \bigg].
\end{equation*}
\end{Lemma}
\begin{proof}
We establish the proof through various steps.\\
\textbf{Step 1}:\ Claim: $\inf_{u \in \Gamma\backslash\{0\} \cap H^+} \|u\|>0.$\\
We argue with contradiction, suppose there exists a sequence $\{u_k\} \subset \Gamma\backslash \{0\} \cap H^+$ such that $\|u_k\| \rightarrow 0$.
 Then using Proposition \ref{HLS} and putting the value of $f(u)= u |u|^p \exp(|u|^{\gamma})$ as well as $f^{'}(u)= ((p+1)+\gamma |u|^{\gamma})|u|^p \exp(|u|^{\gamma})$ we obtain
 %we have,
\begin{align*}
 &|B(u_k)|=\left|\int_{\Omega} \left( \int_{\Omega} \frac{F(u_k)}{|x-y|^\mu |y|^\alpha} ~dy \right) \frac{f'(u_k) u_k^2}{|x|^\alpha} ~dx +\int_{\Omega} \left( \int_{\Omega} \frac{f(u_k) u_k }{|x-y|^\mu |y|^\alpha} ~dy \right) \frac{f(u_k) u_k}{|x|^\alpha} ~dx\right|\\
& \leq C_1 \bigg(\int_{\Omega} (|u_k|^{p+2} \exp(|u_k|^{\gamma}))^{\frac{2n}{2n-(2 \alpha +\mu)}}~dx\bigg)^{\frac{2n-(2 \alpha +\mu)}{n}} +  C_2 \left(\int_{\Omega} (F(u_k))^{\frac{2n}{2n-(2 \alpha +\mu)}}~dx
 \right)^{\frac{2n-(2 \alpha +\mu)}{2n}} \\
 & \quad \times \left(\int_{\Omega}(((p+1)+\gamma |u_k|^{\gamma})|u_k|^{p+2}\exp(|u_k|^{\gamma}))^{\frac{2n}{2n-(2 \alpha +\mu)}}~dx\right)^{\frac{2n-(2 \alpha +\mu)}{2n}},
 \end{align*}
where $C_1,C_2$ are positive constants independent of $u_k$.
Now $(p+2)F(t) \leq tf(t)$ and H\"older's inequality implies that
 \begin{align*}
 |B(&u_k)|\leq C_1\bigg(\int_{\Omega} |u_k|^{\frac{2n \delta'(p+2)}{2n-(2\alpha  +\mu)}}~dx\bigg)^{\frac{2n-(2\alpha +\mu)}{n\delta'}}\times \bigg(\int_{\Omega} \exp\bigg(|u_k|^{\gamma}\frac{2n \delta}{2n-(2\alpha+\mu)}\bigg)~dx\bigg)^{\frac{2n-(2\alpha +\mu)}{n\delta}}\\
 &+  C_2 \bigg(\int_{\Omega} |u_k|^{\frac{2n \delta'(p+2)}{2n-(2\alpha +\mu)}}~dx\bigg)^{\frac{2n-(2\alpha +\mu)}{2n\delta'}}\times \bigg(\int_{\Omega} \exp\bigg(|u_k|^{\gamma} \frac{2n \delta}{2n-(2\alpha +\mu)}\bigg)~dx\bigg)^{\frac{2n-(2\alpha+\mu)}{2n\delta}} \times\\
 &\left[\bigg(\int_{\Omega} |u_k|^{\frac{2n\delta'(p+2)}{2n-(2\alpha  +\mu)}}~dx \bigg)^{\frac{2n-(2\alpha +\mu)}{2n\delta'}} \times \bigg(\int_{\Omega} \exp\bigg(|u_k|^{\gamma} \frac{2n \delta}{2n-(2\alpha +\mu)}\bigg)~dx\bigg)^{\frac{2n-(2\alpha +\mu)}{2n\delta}}\right. \\
 &+ \left.\bigg(\int_{\Omega} |u_k|^{\frac{2n\delta'(p+\gamma+2)}{2n-(2\alpha +\mu)}}~dx\bigg)^{\frac{2n-(2\alpha+\mu)}{2n\delta'}} \times \bigg(\int_{\Omega} \exp\bigg(|u_k|^{\gamma} \frac{2n \delta}{2n-(2\alpha +\mu)}\bigg)~dx\bigg)^{\frac{2n-(2\alpha +\mu)}{2n\delta}}\right],
 \end{align*}
where $\delta >1$ (which depends on $k$) and $\delta^\prime$ denotes its H\"{o}lder conjugate. Using Moser-Trudinger inequality for $u_k$ with large enough $k$ such that $\frac{2n \delta} {(2n-(2\alpha+\mu))}\|u_k\|^\gamma \leq\zeta_{m,2m} $ (such $k$ can be chosen because $\|u_k\| \to 0$ as $k \to \infty$) and $v_k= \frac{u_k}{||u_k||}$, we obtain
 \begin{align*}
 |B(u_k)|&\leq C_1\bigg(\int_{\Omega} |u_k|^{\frac{2n \delta'(p+2)}{2n-(2\alpha+\mu)}}~dx\bigg)^{\frac{2n-(2\alpha  +\mu)}{n\delta'}} \times \bigg(\sup_{\|v_k\| \leq 1}\int_{\Omega} \exp(|v_k|^{\gamma}\zeta_{m,2m})~dx \bigg)^{\frac{2n-(2\alpha  +\mu)}{n\delta}}\\
 &+  C_2 \bigg(\int_{\Omega} |u_k|^{\frac{2n \delta'(p+2)}{2n-(2\alpha  +\mu)}}~dx\bigg)^{\frac{2n-(2\alpha +\mu)}{2n\delta'}} \times\bigg(\sup_{\|v_k\| \leq 1}\int_{\Omega} \exp(|v_k|^{\gamma}\zeta_{m,2m})~dx \bigg)^{\frac{2n-(2\alpha  +\mu)}{n\delta}} \times\\
 &\left[\bigg(\int_{\Omega} |u_k|^{\frac{2n\delta'(p+2)}{2n-(2\alpha +\mu)}}~dx\bigg)^{\frac{2n-(2\alpha +\mu)}{2n\delta'}} + \bigg(\int_{\Omega} |u_k|^{\frac{2n\delta'(p+\gamma+2)}{2n-(2\alpha+\mu)}}~dx\bigg)^{\frac{2n-(2\alpha +\mu)}{2n\delta'}}\right].
 \end{align*}
Finally the Sobolev embedding gives the following upper bound.
 \begin{align*}
 |B(u_k)|&\leq C_3( \|u_k\|^{2(p+2)} +\|u_k\|^{(p+2)}(\|u_k\|^{(p+2)}+\|u_k\|^{(p+\gamma+2)}))\leq C \|u_k\|^{(2p+4)}+ \|u_k\|^{(2p+{\frac{\gamma}{2}}+4)}.
 \end{align*}
Using $u_k \in \Gamma\backslash\{0\}$   we get
$1 \leq C (\|u_k\|^{(2p+1)}+\|u_k\|^{(2p+{\frac{\gamma}{2}}+1)}$, which is a contradiction as $\|u_k\| \to 0$ as $k \to \infty$. Therefore we have $\inf_{u \in \Gamma\backslash\{0\} \cap H^+} \|u\|>0.$\\
\textbf{Step 2}: Claim: $0< \inf_{u \in \Gamma\backslash\{0\}\cap H^+} \left\{\displaystyle\int_{\Omega} \int_{\Omega}\left(\frac{f(u)u}{|x-y|^\mu |y|^\alpha}dy\right)\left(p-2+\gamma |u|^{\gamma}\right)\frac{\exp(|u|^\gamma) |u|^{p+2}}{|x|^\alpha}~dx \right\} $.\\
Since $F(s) \leq \frac{f(s)s}{p+2}$, then by the definition of $\Gamma$ and from Step 1, we obtain $0< \inf_{u \in \Gamma\backslash\{0\}\cap H^+} B(u)$ {\it i.e.}
\begin{align*}
0&< \inf_{u \in \Gamma\backslash\{0\}\cap H^+} \left\{\int_{\Omega} \left( \int_{\Omega} \frac{F(u)}{|x-y|^\mu |y|^\alpha} ~dy \right) \frac{f'(u) u^2}{|x|^\alpha} ~dx +\int_{\Omega} \left( \int_{\Omega} \frac{f(u) u }{|x-y|^\mu |y|^\alpha} ~dy \right) \frac{f(u) u}{|x|^\alpha} ~dx\right\}\\
&\leq \inf_{u \in \Gamma\backslash\{0\}\cap H^+} \left\{\int_{\Omega} \left(\int_{\Omega} \frac{f(u)u}{|x-y|^\mu |y|^\alpha} ~dy\right)\frac{f(u)u+ f'(u)\frac{u^2}{p+2}}{|x|^\alpha} ~dx\right\}\\
&=\inf_{u \in \Gamma\backslash\{0\}\cap H^+} \left\{\int_{\Omega} \left(\int_{\Omega} \frac{f(u)u}{|x-y|^\mu |y|^\alpha} ~dy\right) \frac{|u|^{p+2} exp(|u|^\gamma)}{|x|^\alpha}\bigg(1+ \frac{(p+1)+\gamma |u|^{\gamma}}{p+2}\bigg)\right\}.
\end{align*}
Since $p>2$, we infer $$0< \inf_{u \in \Gamma\backslash\{0\}\cap H^+} \left\{\displaystyle\int_{\Omega} \int_{\Omega}\left(\frac{f(u)u}{|x-y|^\mu |y|^\alpha}\right)\left(p-2+\gamma |u|^{\gamma}\right)\frac{\exp(|u|^\gamma) |u|^{p+2}}{|x|^\alpha} ~dx \right\}.$$
\textbf{Step 3:} Claim: $\Gamma_0>0$. Firstly, we have
\begin{equation}\label{hesti}
H(u)=\int_{\Omega} h(x) |u|^{q+1} \leq \bigg(\int_{\Omega} |h(x)|^\rho\bigg)^{1/\rho} \bigg(\int_\Om|u|^{(1+q)\rho'}\bigg)^{1/{\rho'}}\leq l \|u\|^{q+1}.
\end{equation}
where $l = \|h\|_{L^\rho(\Omega)}$ and $\rho>1$ will be specified later. Choosing
\begin{equation}\label{lamesti}
\la < \frac{2b}{\left(3-q\right)l} M_0 :=\la_0
\end{equation}
where $M_0= \inf_{u \in \Gamma \backslash \{0\} \cap H^+} \|u\|^{1-q} > 0$, we get that $\la l\left(3-q\right) \|u\|^{1+q} < \ 2 b \|u\|^2$ for any $u \in \Gamma \backslash \{0\} \cap H^+$ . Then for $u \in \Gamma\backslash\{0\}\cap H^+$ and $p >2$,
\begin{align*}
&B(u)+2 b\|u\|^2- 3\int_{\Omega} \left(\int_{\Omega}\frac{F(u)}{|x-y|^\mu |y|^\alpha}~dy\right)\frac{ f(u)u}{|x|^\alpha}~dx -\la \left(3-q\right) H(u)\\
&\geq \int_{\Omega} \left(\int_{\Omega}\frac{F(u)}{|x-y|^\mu |y|^\alpha}~dy\right)\frac{f'(u)u^2-3 f(u)u}{|x|^\alpha}~dx + \int_{\Omega}\left( \int_{\Omega}\frac{f(u)u}{|x-y|^\mu |y|^\alpha}~dy\right)\frac{f(u)u}{|x|^\alpha} ~dx \\
& \quad+ 2b \|u\|^2 - \left(3-q\right) \la H(u) > 0.
\end{align*}
Therefore $\Gamma_0 >0$.
\hfill{\QED} \vspace{.2cm}\\
\end{proof}
Next we move on the proof of the claim made in \textbf{Case 2}. From Lemma  \ref{inf} and  \eqref{KC1}, we notice that for $u \in H^+ \backslash \{0\}$, there exists a $t_*>0$, local maximum of $\psi_u$ verifying $\psi_u(t_*)-\la H(u)>0$ since $t_*u \in \Gamma\setminus\{0\} \cap H^+$. From $\psi_u(0)=0$, $\psi_u(t_*)> \la H(u) >0$ and $\lim_{t \rightarrow \infty} \psi_u(t)= -\infty$, there exists $t_1=t_1(u) < t_* < t_2(u)=t_2$ such that $\psi_u(t_1)= \la \int_{\Omega} h(x)|u|^{q+1} ~dx=\psi_u(t_2)$ with $\psi_u'(t_1)>0,\psi_u'(t_2)<0$. Therefore, $t_1u \in N^+_{\la, M}$ and  $t_2u \in N^-_{\la, M}$. Now we show that $t_1 u \in N^+_{\la, M}$ and $t_2 u \in N^-_{\la, M}$ are unique. Suppose not, then there exists $t_3>0$ such that $t_3 u \in N^+_{\la, M}$ and $t_{**}$ such that $t_2 < t_{**} < t_3$,  $\psi_u'(t_{**})=0$ and $\psi_u(t_{**}) <  \la H(u).$ Our Lemma \ref{inf} then induces that if $\psi_u'(t_{**})=0$ then  $\psi_u(t_{**}) > \la H(u)$ which is a contradiction.
% {\color{blue} and finishes the proof of Theorem \ref{ana}.}\\
We will denote $t_*$ as the smallest critical point of $\psi_u$ in the sequel.

\begin{Lemma}
If $\la \in (0,\la_0)$ then $N^0_{\la, M}= \emptyset.$
\end{Lemma}
\begin{proof}
Let $u \in N^0_{\la,M}$  then $u$ satisfies
\begin{align}\label{KC2}
a\|u\|^{4}+ b\|u\|^{2}= \la H(u)+ \int_{\Omega} \int_{\Omega}\left(\frac{F(u)}{|x-y|^\mu |y|^\alpha}dy\right) \frac{f(u)u}{|x|^\alpha} ~dx\mbox{ \ \  and}
\end{align}
\begin{align}\label{kc3}
3 a \|u\|^4+ b \|u\|^2= \la q H(u)+B(u).
\end{align}
Let $u \in H^+ \cap N_{\la, M}^0,$ then  substituting the value $\la H(u)$ from \eqref{KC2} into \eqref{kc3}, we obtain
$$2 \sqrt{\left(3-q\right)\left(1-q\right)ab \|u\|^3}  \leq B(u)$$
which implies $u \in \Gamma \backslash \{0\} \cap H^+.$
Again substituting the value of $a \|u\|^4$ from \eqref{KC2} into \eqref{kc3}, we obtain
\begin{equation*}
B(u)- 3\int_{\Omega} \int_{\Omega}\left(\frac{F(u)}{|x-y|^\mu |y|^\alpha} ~dy\right)\frac{f(u).u}{|x|^\alpha} ~dx +2 b \|u\|^2 - \la \left(3-q\right)H(u)=0
\end{equation*} which contradicts Lemma \ref{inf}.
If $u \in H^-_0 \cap N_{\la, M}^0$ then Case 1 implies that $"1"$ is the only critical point of $\Phi_{u,M}$ and $\Phi^{''}_{u,M}(1)<0$  which is a contradiction to the fact that $u \in N_{\la, M}^0.$
\hfill{\QED}
\end{proof}

\subsection{Energy functional estimates}
In this section we prove that $\mathcal{J}_{\la,M}$ is bounded below on $N_{\la,M}$ and achieves its minimum, with the help of  some estimates on $\theta$, where $\theta =\inf_{u \in N_{\la,M}} \mathcal{J}_{\la,M}(u).$
\begin{Theorem}\label{bdd}
$\mathcal{J}_{\la,M}$ is bounded below and coercive on $N_{\la,M}$. Moreover $\theta \geq - C \la^{\frac{2}{1-q}}$ where $C$ depends on $q,b.$
\end{Theorem}
\begin{proof}
Let $u \in N_{\la, M}$ {\it i.e.} $\Phi_{u,M}^{'}(1)=0.$ Then,
\begin{align*}
\mathcal{J}_{\la,M}(u)
%&= \frac{1}{2}\bigg[\frac{a}{2} \|u\|^{4} + b \|u\|^2\bigg]- \frac{\la}{q+1} H(u) - \frac{1}{2} \int_{\Omega} \int_{\Omega}\left(\frac{F(u)}{|x-y|^\mu |y|^\alpha}~dy\right)\frac{F(u)}{|x|^\alpha}~dx\\
&=a \|u\|^{4}\bigg(\frac{p-2}{4(p+2)}\bigg)+ b \|u\|^{2}\bigg(\frac{p}{2(p+2)}\bigg)- \la\bigg(\frac{p+1-q}{(1+q)(p+2)}\bigg)H(u)\\
&\ \ \ \ \ \ \ \ -\frac{1}{2} \int_{\Omega} \int_{\Omega}\left(\frac{F(u)}{|x-y|^\mu |y|^\alpha}~dy\right)\frac{F(u)-\frac{2f(u)u}{p+2}}{|x|^\alpha}~dx.
\end{align*}
Since $0 \leq F(u) \leq \frac{2}{p+2} f(u)u$ and $q<1$,  \eqref{hesti} and Sobolev embedding implies that $\,\mathcal{J}_{\la,M}$ is coercive on $N_{\la, M}$ that is  as $\|u\| \rightarrow \infty$,
\begin{align*}
\mathcal{J}_{\la,M}(u) \geq a \|u\|^{4}\bigg(\frac{p-2}{4(p+2)}\bigg)+ b \|u\|^{2}\bigg(\frac{p}{2(p+2)}\bigg)- \la l \bigg(\frac{p+1-q}{(1+q)(p+2)}\bigg)\|u\|^{q+1} \to \infty.
\end{align*}
Similarly, we have
\begin{align*}
\mathcal{J}_{\la,M} (u)&= \frac{b}{2} \|u\|^2 -\frac{\la}{q+1} H(u)- \frac{1}{2} \int_{\Omega} \int_{\Omega}\left(\frac{F(u)}{|x-y|^\mu |y|^\alpha}~dy\right) \frac{F(u)}{|x|^\alpha}~dx \\
&\ \ \ \ \ + \frac{1}{4}\bigg(\la H(u) + \int_{\Omega} \int_{\Omega}\left(\frac{F(u)}{|x-y|^\mu |y|^\alpha}~dy\right)\frac{f(u)u}{|x|^\alpha} ~dx - b\|u\|^2 \bigg)\\
& \geq \frac{1}{4} b \|u\|^2- \la \bigg(\frac{1}{q+1}-{\frac{1}{4}\bigg)} H(u).
\end{align*}
Then for $u \in H_0^-$, we get $\mathcal{J}_{\la,M}(u) \geq 0$ and for $u \in H^+$, the Sobolev embedding implies
\begin{align*}
\mathcal{J}_{\la,M} (u) &\geq \frac{b}{4} \|u\|^2 - \frac{\la {(3-q)}}{4(q+1)} H(u) \geq \frac{b}{4} \|u\|^2 - \frac{\la {(3-q)}l}{4(q+1)} \left(\int_{\Omega} |u|^{(1+q)\rho^\prime} ~dx \right)^{1/\rho^\prime}\\
&= b_3 \|u\|^{2} - b_4 \|u\|^{q+1}
\end{align*}
where $b_3 = \frac{b}{4} $ and $b_4 = \frac{\la {(3-q)}}{4(q+1)}$. So by finding the minimum of function $g(x)= b_3 x^2 - b_4 x^{q+1}$, we can conclude that $\mc J_{\la,M}$ is bounded below on $N_{\la,M}$.
%Therefore,
%\begin{align*}
%\inf_{u \in N_{\la,M}} \mathcal{J}_{\la,M} (u) &\geq g \left(\frac{ b_4(q+1)}{2b_3 }\right)^{\frac{1}{1-q}}\\
%&= \left(\frac{b_4^2}{b_3^{(q+1)}}\right)^{\frac{1}{1-q}} \left( \left(\frac{(q+1)}{2}\right)^{\frac{2}{1-q}}- \left(\frac{(q+1)}{2}\right)^\frac{(q+1)}{1-q}\right)
%\end{align*}
%which gives $\theta \geq - C(q,b) \la^{\frac{2}{1-q}} $
%where $C(q,b)>0$ and completes the proof of Theorem \ref{bdd}.
\hfill{\QED}
\end{proof}
\vspace{.1cm}
\begin{Lemma}\label{lemmaq}
There exists a constant $C_0>0$ such that $\theta \leq - C_0$.
\end{Lemma}
\begin{proof}
Let $u\in H^+$, then from the fibering map analysis we know that there exists a $t_1(u)>0$ such that $t_1 u \in N^+_{\la ,M}\cap H^+$ and $\psi_{u,M}(t_1)= \la H(u)$. Since $\Phi^{''}_{u,M}(t_1)>0$, from \eqref{KC21} we obtain
\begin{equation}\label{new}
\begin{split}
\frac{q-3}{m} \;a \|t_1 u\|^{4}
{<} \left(1-q\right) b \|t_1 u\|^{2} - B(t_1u)+ q \int_{\Omega} \left(\int_{\Omega} \frac{F(t_1 u)}{|x-y|^\mu |y|^\alpha} ~dy\right) \frac{f(t_1 u) t_1u}{|x|^\alpha} ~dx.
\end{split}
\end{equation}
Using $\Phi^{'}_{u,M}(t_1)=0$, we get that
\begin{align*}
\mathcal{J}_{\la,M}(t_1 u) &= \frac{1}{2}\bigg(\frac{a}{2}\|t_1 u\|^{4} + b\|t_1 u\|^2\bigg)-\frac{1}{2} \int_{\Omega}  \left(\int_{\Omega} \frac{F(t_1 u)}{|x-y|^\mu |y|^\alpha} ~dy\right)\frac{F(t_1 u)}{|x|^\alpha}~dx\\
&\ \ \ \ -\frac{1}{q+1} \bigg( a\|t_1u\|^{4}+b\|t_1u\|^2-\int_{\Omega} \left(\int_{\Omega} \frac{F(t_1 u)}{|x-y|^\mu |y|^\alpha} ~dy\right)\frac{f(t_1u)t_1u}{|x|^\alpha}~dx \bigg).
\end{align*}
In that case, by \eqref{new} we obtain,
\begin{align*}
\mathcal{J}_{\la,M}(t_1 u) &= \frac{-(1-q)}{4(q+1)} b \|t_1 u\|^2+ \int_{\Omega}\left(\int_{\Omega} \frac{F(t_1 u)}{|x-y|^\mu |y|^\alpha} ~dy\right) \bigg(\frac{4+q}{4(q+1)} \frac{f(t_1u)t_1u}{|x|^\alpha} \nonumber\\
&\ \   - \frac{1}{2} \frac{F(t_1u)}{|x|^\alpha}-\frac{f'(t_1u)(tu)^2}{4(q+1)|x|^\alpha}\bigg)~dx -\frac{1}{4(q+1)} \int_{\Omega} \left(\int_{\Omega} \frac{f(t_1 u)t_1 u}{|x-y|^\mu |y|^\alpha} ~dy\right) \frac{f(t_1 u)t_1 u}{|x|^\alpha} ~dx\nonumber\\
& \leq \frac{-(1-q)}{4(q+1)} b \|t_1 u\|^2 + \int_{\Omega} \left(\int_{\Omega} \frac{F(t_1 u)}{|x-y|^\mu |y|^\alpha} ~dy\right)\bigg(\frac{4+q}{4(q+1)} -\frac{(p+2)}{4(q+1)}\nonumber\\
&\ \ \ \ \ \ \ - \frac{(p+1)}{4(q+1)}\bigg)\frac{f(t_1u)t_1u}{|x|^\alpha} ~dx- \frac{1}{2} \int_{\Omega} \left(\int_{\Omega} \frac{F(t_1u)}{|x-y|^\mu |y|^\alpha} ~dy\right) \frac{F(t_1 u)}{|x|^\alpha} ~dx.
\end{align*}
Since $ 1+q-2p <0$ therefore $\theta \leq \inf_{u \in N_{\la,M}^+ \cap H^+} \mathcal{J}_{\la,M}(u)\leq - C_0 <0.$
\hfill{\QED}\vspace{.2cm}\\
\end{proof}
Using Theorem \ref{bdd} and Ekeland variational principle, we know that  there exists a sequence $\{u_k\}_{k \in \mathbb{N}} \subset N_{\la, M}$ such that
\begin{equation}\label{Ekeland}
     \left\{
         \begin{alignedat}{2}
             {}    \mathcal{J}_{\la,M}(u_k)
             & {} \leq \theta+ \frac{1}{k};
             \\
              \mathcal{J}_{\la,M}(v)
             & {}\geq \mathcal{J}_{\la,M}(u_k) -\frac{1}{k}\|u_k-v\|,
             && \ \ \ \forall v \in N_{\la, M}.
        \end{alignedat}
     \right.
\end{equation}
Then by \eqref{Ekeland} and Lemma \ref{lemmaq}, we have for large $k$,\\\
\begin{equation}\label{KC9}
\mathcal{J}_{\la,M}(u_k) \leq - \frac{C_0}{2}.
\end{equation}
Also since $u_k \in N_{\la,M}$ we have
\begin{align*}
\mathcal{J}_{\la,M}(u_k)&=a \|u_k\|^{4}\bigg(\frac{p-2}{4(p+2)}\bigg)+ b \|u_k\|^{2}\bigg(\frac{p}{2(p+2)}\bigg)- \la\bigg(\frac{p+1-q}{(1+q)(p+2)}\bigg)H(u_k)\\
&\ \ \ \ \ \ \ \ -\frac{1}{2} \int_{\Omega} \left(\int_{\Omega} \frac{F(u_k)}{|x-y|^\mu |y|^\alpha} ~dy\right) \frac{F(u_k)-\frac{2f(u_k)u_k}{p+2}}{|x|^\alpha}~dx.
\end{align*}
This together with \eqref{KC9} gives
$$- \la \bigg(\frac{p+1-q}{(1+q)(p+2)}\bigg)H(u_k) \leq - \frac{C_0}{2} \Longrightarrow H(u_k) \geq \frac{C_0(p+2)(1+q)}{2\la(p+1-q)} > 0$$
{\it i.e.}
\begin{equation}\label{sequence}
 H(u_k) > C >0,\; \text{for large} \;k \ \ \text{and}\ \ u_k \in N_{\la,M} \cap H^+.
\end{equation}
The following result shows that minimizers for $\mathcal{J}_{\la,M}$ in any subset of the decomposition of $N_{\la,M}$ are critical points of $\mathcal{J}_{\la,M}$ and the proof follows from the Lagrange multipliers rule (see Lemma 4.7 in  \cite{AGMS}).

\begin{Lemma}
Let u be a local minimizer for $\mathcal J_{\la,M}$ on any subsets of $N_{\la,M}$ such that $u \not\in N_{\la,M}^0$. Then $u$ is a critical point of $\mathcal{J}_{\la,M}.$
\end{Lemma}
%\begin{proof}
%Let $u$ be a local minimizer for $\mathcal{J}_{\la, M}.$ Then, in any case $u$ is a minimizer for $\mathcal{J}_{\la,M}$ under the constraint $I_{\la,M}(u):=\langle \mathcal{J}'_{\la,M}(u),u \rangle =0.$ Hence, by the theory of Lagrange multipliers there exists a $\mu \in \mathbb{R}$ such that $\mathcal{J}'_{\la,M} = \mu I_{\la,M}'(u).$ Thus $\langle \mathcal{J}'_{\la,M}(u), u \rangle = \mu \langle I'_{\la,M}(u),u \rangle = \mu \Phi''_{\la, M}(1)=0,$ but $u \not\in N^0_{\la,M}$ and so $\Phi''_{\la, M}(1) \neq 0.$ Hence $\mu=0.$
%\hfill{\QED}
%\end{proof}

\begin{Lemma}\label{compl1}
Let $\la >0$ satisfies \eqref{lamesti}. Then for any $u \in N_{\la, M} \backslash \{0\},$ there exists a $\epsilon >0$ and a differentiable function $\xi : B(0, \epsilon) \subset W_0^{m,2}(\Omega) \rightarrow \mathbb{R}$ such that $$\xi(0)=1 \ \text{and}\  \xi(w)(u-w) \in N_{\la,M}$$for all $w \in W_0^{m,2}(\Omega)$. Moreover
{\begin{align*}\label{4.10}
\langle &\xi'(0),w \rangle = \frac{2(2a\|u\|^2 +b) \int_{\Omega}  \nabla^m u.\nabla^m w ~dx- \la (q+1) \int_{\Omega} h(x) |u|^{q-1} u w ~dx - \langle S(u),w \rangle}{a \left(3-q\right)\|u\|^4+ b\left(1-q \right) \|u\|^2 + R(u)}
\end{align*}
where $$R(u)=\int_{\Omega} \left( \int_{\Omega} \frac{F(u)}{|x-y|^\mu |y|^\alpha} \right) \frac{qf(u)- f'(u)u)u}{|x|^\alpha}~dx -\int_{\Omega} \left(\int_{\Omega}  \frac{f(u)u}{|x-y|^\mu |y|^\alpha} ~dy\right) \frac{f(u)u}{|x|^\alpha}~dx$$
and $$\langle S(u),w \rangle= \int_{\Omega}  \left( \int_{\Omega} \frac{F(u)}{|x-y|^\mu |y|^\alpha} ~dy\right) \frac{f'(u)u+ f(u)}{|x|^\alpha} w ~dx + \int_{\Omega}  \left(\int_{\Omega}  \frac{f(u)u}{|x-y|^\mu |y|^\alpha} ~dy\right) \frac{f(u)}{|x|^\alpha}w~dx .$$}
\end{Lemma}
\begin{proof}
For $u \in N_{\la, M} ,$ we define a continuous differentiable function $G_u: \mathbb{R}\times W_0^{m, 2}(\Omega) \rightarrow \mathbb{R}$ such that
\begin{align*}
G_u(t,v)&=a t^{3-q}\|u-v\|^{4}+ b t^{1-q} \|u-v\|^{2} - \frac{1}{t^q} \int_{\Omega} \left(\int_{\Omega} \frac{F(t(u-v))}{|x-y|^\mu |y|^\alpha}~dy \right) \frac{f(t(u-v))(u-v)}{|x|^\alpha} ~dx \\
&-{\la  \int_{\Omega} h(x)|u-v|^{q+1} ~dx}.
\end{align*}
Then $G_u(1,0)= \Phi_u'(1)=0 $
and
$\displaystyle\frac{\partial}{\partial t} G_u(1,0)
%&= a\left(3-q\right)\|u\|^{4} + b\left(1-q\right)\|u\|^{2} + q \int_{\Omega} \left(\int_{\Omega} \frac{F(u)}{|x-y|^\mu |y|^\alpha}~dy \right) \frac{f(u).u}{|x|^\alpha} ~dx \\
%&\ \ \ - B(u)
= \phi_u''(1)\neq 0.$
Hence by the implicit function theorem, there exists $\epsilon >0$ and a differentiable function $\xi : B(0, \epsilon) \subset W_0^{m, 2}(\Omega) \rightarrow \mathbb{R}$ such that $\xi(0)=1$ and $G_u(\xi(w),w)=0 \ \ \forall w \in B(0, \epsilon)$ which is equivalent to $\langle \mathcal{J}_{\la,M}'(\xi(w)(u-w)),\xi(w)(u-w)\rangle =0 \ \ \forall \ w \in B(0, \epsilon)$. Thus, $\xi(w)(u-w) \in N_{\la,M}$ and differentiating $G_u(\xi(w),w)=0$
%\begin{align*}
%G_u(\xi(w),w)&= a (\xi(w))^{3-q}\|u-w\|^4 + b (\xi(w))^{1-q} %\|u-w\|^2 -\la \int_{\Omega} h(x)|u-w|^{q+1} \\
%&-\frac{1}{(\xi(w))^q} \int_{\Omega} \left(\int_{\Omega} \frac{F(\xi(w))(u-w))}{|x-y|^\mu |y|^\alpha}~dy \right) \frac{f(\xi(w)(u-w))(u-w)}{|x|^\alpha} ~dx=0
%\end{align*}
with respect to $w$, we obtain the required claim.
%\begin{align}
%\langle &\xi'(0), \rangle = \frac{n(2a\|u\|^n+b) \int_{\Omega} |\nabla(u)|^{n-2} \nabla u.\nabla w- \la (q+1) \int_{\Omega} h(x) |u|^{q-1} u w - \langle S(u),w \rangle} {a(2n-1-q)\|u\|^{2n}+ b(n-1-q)\|u\|^n + R(u)}
%\end{align}
%where $$R(u)=\int_{\Omega} (|x|^{-\mu}*F(u))(qf(u)-f'(u).u).u -\int_{\Omega} (|x|^{-\mu}*f(u).u)f(u)u$$
%and $$\langle S(u), w\rangle= \int_{\Omega} (|x|^{-\mu}*F(u))(f'(u)u +f(u)) w + \int_{\Omega} (|x|^{-\mu}*f(u)u)f(u)w.$$
\hfill{\QED} 
\end{proof}

\noi Similarly, by following the  proof of Lemma $4.9$ of \cite{AGMS} and using Lemma \ref{compl1}, we have the following result.

\begin{Lemma}\label{compli2}
Let $\la >0$ satisfies \eqref{lamesti} then given any $ u \in N^-_{\la,M} \backslash \{0\},$ then there exists $\epsilon >0$ and a differentiable function $\xi^- : B(0, \epsilon) \subset W_0^{m, 2}(\Omega) \rightarrow \mathbb{R}$ such that $$\xi^-(0)=1 \ \text{and}\ \xi^-(w)(u-w) \in N^-_{\la,M}$$ and for all $w \in W_0^{m, 2}(\Omega)$
{ \begin{align*}
\langle &(\xi^-)'(0),w \rangle = \frac{2(2a\|u\|^2 +b) \int_{\Omega}  \nabla^m u.\nabla^m w ~dx- \la (q+1) \int_{\Omega} h(x) |u|^{q-1} u w ~dx - \langle S(u),w \rangle} {a \left(3-q\right)\|u\|^4 + b\left(1-q \right) \|u\|^2  + R(u)}
\end{align*}}
where $R(u)$ and $S(u)$ are as in lemma~\ref{compl1}.
%$$R(u)=\int_{\Omega} (|x|^{-\mu}*F(u))(qf(u)-f'(u).u).u -\int_{\Omega} (|x|^{-\mu}*f(u).u)f(u)u$$%
%and $$\langle S(u),w \rangle = \int_{\Omega} (|x|^{-\mu}*F(u))(f'(u)u +f(u)) w + \int_{\Omega} (|x|^{-\mu}*f(u)u)f(u)w. $$
\end{Lemma}
%%For any $u \in N_{\la,M}^-$, $\Phi^{'}_{u,M} (1)=0$ and $ \Phi^{''}_{u,M} (1)<0$.
%\begin{equation}
%\Phi^{'}_{u,M} (t) = t^{n-1} \|u\|^n m(\|tu\|^n) - \la t^{q} \int_{\Omega} h(x) |u|^{q+1}~dx - \int_{\Omega} (|x|^{-\mu}* F(tu)) f(tu) u ~dx =0
%\end{equation}
%and
%\begin{align*}
    %  \Phi^{''}_{u,M} (t)&= n t^{n-2} \|u\|^{2n} m(\|tu\|^n)+ (n-1) t^{n-2} \|u\|^n m(\|tu\|
    %  ^n) - \la q t^{q-1} \int_{\Omega} h(x) |u|^{q+1}~dx\\
    %  &- \int_{\Omega} (|x|^{-\mu} *f(tu).u)f(tu)u ~dx - \int_{\Omega} (|x|^{-\mu}* F(tu))f'(tu) u^2 ~dx
     % <0.
%\end{align*}
%%This implies $ u \in \Gamma\backslash \{0\}$.  Then by Lemma \ref{compl1} there exists $\epsilon >0$ and a differentiable function $\xi^-:B(0, \epsilon)\subset W_0^{m,2}(\Omega) \rightarrow \mathbb{R}$ such that $\xi^-(0)=1$, and $\xi^-(w)(u-w) \in N_{\la,M}$ for all $w \in B(0, \epsilon).$ Then by the continuity of $\mathcal{J}^{'}_{\la,M}$ and $\xi^-$ and by choosing $\epsilon$ small enough we have
%%\begin{align*}
     % (u-w)\|^2 M(\|tu\|^2) \\
      %&- \la q \int_{\Omega} h(x) |\xi^-(u)(u-w)|^{q+1}~dx\\
     % &- \int_{\Omega} \left(\frac{f(\xi^-(u)(u-w))\xi^-(u)(u-w)}{|x-y|^\mu |y|^\alpha}\right).
      %\frac{f(\xi^-(u)(u-w))\xi^-(u)(u-w)}{|x|^\alpha} ~dx \\
      %&- \int_{\Omega} \left(\int_{\Omega} \frac{F(\xi^-(u)(u-w))}{|x-y|^\mu |y|^\alpha}\right) \frac{f'(\xi^-(u)(u-w))(\xi^-(u)(u-w))^2}{|x|^\alpha} ~dx
     % <0
%\end{align*}
% implies $\xi^-(w)(u-w) \in N_{\la,M}^-.$
Now we prove the following result.
\begin{Proposition}\label{j}
Let $\la >0$ satisfies \eqref{lamesti} and $u_k\in N_{\la,M}$  satisfies \eqref{Ekeland}. Then $\|\mathcal{J}^{'}_{\la,M}(u_k)\|_* \rightarrow 0$ as $k \rightarrow \infty.$
\end{Proposition}
\begin{proof}
 Step 1: $\liminf_{k \rightarrow \infty} \|u_k\| >0.$\\
\ \ We know that $\{u_k\}$ satisfies \eqref{sequence} for large $k$, thus $H(u_k)\geq C >0$ for large $k$. So by using H\"older inequality we obtain
$C < \ H(u_k) \leq C_1 \|u_k\|^{q+1}$.\\
Step 2: We claim that
$$\liminf_{k \rightarrow \infty} \left[\left(3-q\right)a \|u_k\|^{4} + b\left(1-q\right) \|u_k\|^2 +q \int_{\Omega} \left(\int_{\Omega}\frac{F(u_k)}{|x-y|^\mu |y|^\alpha}  \right)\frac{f(u_k)u_k}{|x|^\alpha} ~dx - B(u_k)\right] >0.$$
{Without loss of generality, we can assume that $u_k \in N^+_{\la,M}$ (if not replace $u_k$ by $t_1(u_k)u_k$).} Arguing by contradiction, suppose that there exists a subsequence of $\{u_k\}$, still denoted by $\{u_k\}$, such that
$$0\leq \left(3-q\right)a \|u_k\|^{4} + b\left(1-q\right) \|u_k\|^2 +q \int_{\Omega} \left(\int_{\Omega}\frac{F(u_k)}{|x-y|^\mu |y|^\alpha} \right)\frac{f(u_k)u_k}{|x|^\alpha} ~dx - B(u_k)= o_k(1).$$
From Step $1$ and the above equation we obtain that $\liminf_{k \rightarrow \infty} B(u_k) >0$ and
$$\left(3-q\right)a \|u_k\|^{4} + b\left(1-q\right) \|u_k\|^2 \leq B(u_k)$$ {\it i.e.} $u_k \in \Gamma\backslash \{0\} \ $ for all large $k.$\\
Since {$u_k \in N^+_{\la,M}$}, we get
\begin{align*}
-2 b\|u_k\|^2+ \la \left(3-q \right) H(u_k) +3 \int_{\Omega} \left(\int_{\Omega}\frac{F(u_k)}{|x-y|^\mu |y|^\alpha} \right) \frac{f(u_k)u_k}{|x|^\alpha} ~dx - B(u_k)&= o_k(1)
\end{align*}
which is a contradiction since $\Gamma_0 >0.$
The remaining proof follows similarly as the proof of Proposition $4.10$ of \cite{AGMS}.
\hfill{\QED}
\end{proof}
\subsection{Existence of local minimum of $\mathcal{J}_{\lambda, M}$ in ${N}_{\lambda, M}$}
\begin{Theorem}\label{exis1}
Let $1<\gamma < 2$ and  $\la >0$ satisfies \eqref{lamesti}. Then there exists a weak solution  $u_{\la} \in N^+_{\la,M} \cap H^+$ to $({\mathcal P}_{\lambda,M})$ such that $\mathcal{J}_{\la, M}(u_\la)= \inf_{u\in N_{\la, M}\backslash \{0\}} \mathcal{J}_{\la, M}(u)$ and $u_{\la} \in N^+_{\la,M} \cap H^+$ is a local minimum for $\mathcal{J}_{\la, M}$ in $W_0^{m, 2}(\Omega).$
\end{Theorem}
\begin{proof}
Let $\{u_k\} \subset \mathcal{N}_{\lambda,M}$ be a minimizing sequence satisfying $\mathcal{J}_{\la, M}(u_k) \rightarrow \theta$ as $k \rightarrow \infty$ and $\mathcal{J}_{\la, M}(v) \geq \mathcal{J}_{\la, M}(u_k) - \frac{1}{k} \|u_k-v\|,   \ \forall v \in N_{\la}$ (as in \eqref{Ekeland}). Then by Theorem \ref{bdd} we obtain $\{u_k\}$ is a bounded sequence in $W_0^{m,2}(\Omega).$ Also there exists a subsequence of $\{u_k\}$ (denoted by same sequence) and $u_{\la}$ such that $u_k \rightharpoonup u_\la $ weakly in $W_0^{m, 2}(\Omega)$, $u_k \rightarrow u_{\la}$ strongly in $L^r(\Omega)\ $ for $ r \geq 1$ and $u_k \rightarrow u_\la$  a.e. in $\Omega$ as $k \to \infty.$ Then using $f(t) \leq C_{\epsilon,\gamma} \exp(\epsilon t^{2})$ for $\epsilon>0$ small enough and Theorem \ref{TM-ineq} with $n=2m$,
we obtain that $f(u_k)$ and $F(u_k)$ are uniformly bounded in $L^q(\Omega)$ for all $q>1.$ Then by Proposition \ref{HLS} and  Vitali's convergence theorem, we obtain
\begin{align*}
\left|\int_{\Omega}  \left(\int_{\Omega}\frac{F(u_k)}{|x-y|^\mu |y|^\alpha}~dy \right) \frac{f(u_k) (u_k- u_\la)}{|x|^\alpha}~dx\right| \rightarrow 0 \ \text{as} \ \ k \rightarrow \infty.
\end{align*}
Thus by Proposition \ref{j}, we have $\langle \mathcal{J}^{'}_{\la, M}(u_k),(u_k- u_\la)\rangle \rightarrow 0.$ Then we conclude that
\begin{equation}\label{KC11}
M(\|u_k\|^2) \int_{\Omega} \nabla^m u_k. \nabla^m(u_k- u_{\la}) ~dx \rightarrow 0 \; \text{as}\; k\to \infty.
\end{equation}
On the other hand, using $u_k \rightharpoonup u_\la$ weakly and by boundedness of $M(\|u_k\|^2)$ we have
\begin{equation}\label{KC12}
M(\|u_k\|^2) \int_{\Omega} \nabla^m u_\la. \nabla^m(u_k- u_{\la}) ~dx \rightarrow 0\; \text{as}\; k \to \infty.
\end{equation}
Substracting \eqref{KC12} from \eqref{KC11}, we get,
\begin{equation*}
M(\|u_k\|^2) \int_{\Omega} ( \nabla^m u_k-\nabla^m u_\la). \nabla^m(u_k- u_{\la}) ~dx \rightarrow 0\; \text{as}\; k \to \infty.
\end{equation*}
which gives
\begin{equation*}
M(\|u_k\|^2) \int_{\Omega} |\nabla^m u_k - \nabla^m u_\la|^2 ~dx \rightarrow 0 \ \ \text{as} \ k \rightarrow  \infty.
\end{equation*}
Since $M(t) \geq M_0$, we obtain $u_k \rightarrow u_{\la}$ strongly in $W_0^{m, 2}(\Omega)$. By Lemma \ref{kc-ws}
\begin{equation*}
\int_{\Omega}\left(\int_\Om\frac{F(u_k)}{|x-y|^\mu |y|^\alpha} dy\right)\frac{f(u_k)}{|x|^\alpha} \phi ~dx \rightarrow \int_{\Omega}\left(\int_\Om\frac{F(u_\la)}{|x-y|^\mu |y|^\alpha} dy\right)\frac{f(u_\la)}{|x|^\alpha} \phi ~dx
\end{equation*}
and also
\begin{equation*}
\int_{\Omega} h(x) |u_k|^{q-1} u_k \phi ~dx \rightarrow \int_{\Omega} h(x) |u|_\la^{q-1} u_\la \phi ~dx
\end{equation*}
for all $\phi \in W_0^{m, 2}(\Omega)$.
Therefore, $u_\la$ satisfies $(\mathcal{P}_{\la,M})$ in weak sense and hence $u_\la \in N_{\la, M}.$ Moreover, $\theta \leq \mathcal{J}_{\la,M}(u_\la) \leq \liminf_{k \rightarrow \infty} \mathcal{J}_{\la, M}(u_k) = \theta.$ Hence $u_\la$ is a minimizer for $\mathcal{J}_{\la, M}$ in $N_{\la, M}$.\\
Using \eqref{sequence}, we have $\int_{\Omega} h(x) |u_{\la}|^{q+1} >0$. Then there exists a $t_1(u_\la)>0$ such that $t_1(u_\la) u_\la \in N^{+}_{\la, M}$. We now claim that $t_1(u_\la)=1$ {\it i.e.} $u_\la \in N^+_{\la, M}.$ Suppose not then $t_2(u_\la)=1$ and $u_\la \in N^{-}_{\la, M}.$ Now $\mathcal{J}_{\la,M}(t_1(u_\la)u_\la) < \mathcal{J}_{\la, M}(u_\la) \leq \theta $ which yields a contradiction, since $t_1(u_{\la}) u_\la \in N_{\la,M}.$
The proof for $u_\la$ being a local minimum for $\mc J_{\la,M}$ in $W^{m,2}_0(\Om)$ follows exactly as the proof of Theorem $4.12$ in \cite{AGMS}.
%Thus, $u_\lambda$ is nontrivial.
% From the strong comparison principle, we get $u_\lambda>0$ in $\Om$.
\hfill{\QED}
\end{proof}

\begin{Theorem}\label{exis3}
Let $1<\gamma < 2$ and $\la >0$ satisfies \eqref{lamesti}. Then $\mathcal{J}_{\la, M}$ achieves its minimizer on $N^-_{\la, M}.$
\end{Theorem}
\begin{proof}
Let $u \in N^-_{\la,M}$. Then
\begin{align*}
3a \|u\|^{4} &+ b \|u\|^2- \la q H(u)- \int_{\Omega} \left(\int_{\Omega}\frac{f(u) u}{|x-y|^\mu |y|^\alpha} ~dy \right)\frac{f(u).u}{|x|^\alpha} ~dx\\
& \ \ \ \ \ \ \ \ -\int_{\Omega} \left(\int_{\Omega}\frac{F(u)}{|x-y|^\mu |y|^\alpha} ~dy \right) \frac{f'(u) u^2}{|x|^\alpha} ~dx <0.
\end{align*}
This along with \eqref{ndef} gives us
\begin{equation*}\label{4.39}
\begin{split}
\left(3-q\right)& a \|u\|^{4} + \left(1-q\right) b \|u\|^2+ q \int_{\Omega} \left(\int_{\Omega}\frac{F(u)}{|x-y|^\mu |y|^\alpha} ~dy \right)\frac{f(u.u}{|x|^\alpha} ~dx \\
&- \int_{\Omega} \left(\int_{\Omega}\frac{f(u) u}{|x-y|^\mu |y|^\alpha} ~dy \right)\frac{f(u)u}{|x|^\alpha} ~dx -\int_{\Omega} \left(\int_{\Omega}\frac{F(u)}{|x-y|^\mu |y|^\alpha} ~dy \right) \frac{f'(u) u^2}{|x|^\alpha} ~dx <0.
\end{split}
\end{equation*}
This implies that $N_{\la,M}^-\subset \Gamma$ and then following step $1$ of Lemma \ref{inf} we get that $\exists\ c>0,\ \|u\| \geq c>0$ for any $u \in N^-_{\la,M}$ from which it follows that $N^-_{\la,M}$ is a closed set. Also this gives $\inf_{u \in N^-_{\la, M}\backslash \{0\}} B(u) \geq \tilde{c} >0.$ Therefore, for $\la < \la_0$ small enough,
% Using $p+2 >q$ it is easy to deduce from  \eqref{4.39} that $\exists\ c>0,\ \|u\| \geq c>0$ for any $u \in N^-_{\la,M}\backslash \{0\}$ from which it follows that $N^-_{\la,M}\backslash \{0\}$ is a closed set. Also as in Lemma \ref{inf} we can prove that $N_{\la, M}^- \subset \Gamma$ and then $\inf_{u \in N^-_{\la, M}\backslash \{0\}} B(u) \geq \tilde{c} >0.$ Therefore, for $\la < \la_0$ small enough,
%
\begin{equation*}\label{xi-argument}
\displaystyle\inf_{u\in N^-_{\la,M}\backslash \{0\}}B(u)+ 2 b\|u\|^2 -\left(3-q\right)\lambda H(u)-3 \int_\Om \left(\int_{\Omega} \frac{F(u)}{|x-y|^\mu |y|^\alpha} ~dy\right) \frac{f(u)u}{|x|^\alpha} ~dx>0.
\end{equation*}
Now let $\theta^- = \min_{u \in N^-_{\la, M}\backslash \{0\}} \mathcal{J}_{\la, M}(u) > - \infty$ then from Ekeland variational principle, we know that there exist $\{v_k\}_{k\in \mathbb{N}}$ a minimizing sequence such that
$$\mathcal{J}_{\la, M}(v_k) \leq \inf_{u \in N^-_{\la, M}} \mathcal{J}_{\la,M}(u) + \frac{1}{k} \ \text{and} \ \mathcal{J}_{\la, M}(u) \geq \mathcal{J}_{\la, M}(v_k) - \frac{1}{k}\|v_k-u\|\ \ \forall \ u \in N^-_{\la, M}.$$ From $\mathcal{J}_{\la,M}(v_k) \to \theta^-$ as $k\to\infty$ and $v_k \in N_{\la, M}$, it is easy to prove  that $\|v_k\| \leq C$ (as in Lemma \ref{kc-PS-bdd}).
Indeed,
\begin{align*}
\left| a\|v_k\|^{4}+b \|v_k\|^2 -\lambda H(v_k)- \int_\Om \left(\int_{\Omega} \frac{F(v_k)}{|x-y|^\mu |y|^\alpha} ~dy\right) \frac{f(v_k)v_k}{|x|^\alpha} ~dx \right| =o(\|v_k\|)
\end{align*}
and
\begin{align*}
& C+o(\|v_k\|)\geq \mathcal{J}_{\la,M}(v_k)- {\frac{1}{4}} \langle \mathcal{J}^{'}_{\la, M}(v_k),v_k \rangle
\geq {\frac{b}{4}} \|v_k\|^{2n} -C(\lambda)\|v_k\|^{q+1}
\end{align*}
implies that $\|v_k\| \leq C.$ Thus we get $\|S(v_k)\|_* \leq C_1$ and from \eqref{xi-argument} we have $\|\xi_k^-(0)\|_* \leq C_2$. Now the rest of the proof follows as in the proof of Theorem \ref{exis1} with the help of Lemma~\ref{compli2} (refer Theorem 4.13 of \cite{AGMS}).
\hfill{\QED} 
\end{proof}\\
\textbf{Proof of Theorem \ref{first}} : The proof follows from Theorem \ref{exis1} and Theorem \ref{exis3}.
\hfill{\QED}

\end{document}